\documentclass[11pt]{amsart}
\usepackage{amssymb, amstext, amscd, amsmath, relsize}
\usepackage[all]{xy}
\usepackage{centernot}
\usepackage{mathtools}
\usepackage{ stmaryrd }
\usepackage{tikz}

\usetikzlibrary{positioning}

\makeatletter
\def\@cite#1#2{{\m@th\upshape\bfseries%
[{#1\if@tempswa{\m@th\upshape\mdseries, #2}\fi}]}}
\makeatother
\theoremstyle{plain}
\newtheorem{theorem}{Theorem}[section]
\newtheorem{corollary}[theorem]{Corollary}
\newtheorem{proposition}[theorem]{Proposition}
\newtheorem{lemma}[theorem]{Lemma}

\theoremstyle{definition}
\newtheorem{definition}[theorem]{Definition}
\newtheorem{example}[theorem]{Example}
\newtheorem{remark}[theorem]{Remark}
\theoremstyle{remark}
\newcommand{\bbC}{{\mathbb{C}}}

\newcommand{\bbI}{{\mathbb{I}}}

\newcommand{\bbN}{{\mathbb{N}}}
\newcommand{\bbQ}{{\mathbb{Q}}}
\newcommand{\bbR}{{\mathbb{R}}}
\newcommand{\bbT}{{\mathbb{T}}}
\newcommand{\bbZ}{{\mathbb{Z}}}

\newcommand{\A}{{\mathcal{A}}}
\newcommand{\B}{{\mathcal{B}}}
\newcommand{\C}{{\mathcal{C}}}

\newcommand{\G}{{\mathcal{G}}}
\renewcommand{\H}{{\mathcal{H}}}
\newcommand{\I}{{\mathcal{I}}}
\newcommand{\J}{{\mathcal{J}}}
\newcommand{\K}{{\mathcal{K}}}
\renewcommand{\L}{{\mathcal{L}}}

\newcommand{\T}{{\mathcal{T}}}

\newcommand{\X}{{\mathcal{X}}}

\newcommand{\fA}{{\mathfrak{A}}}

\newcommand{\fG}{{\mathfrak{G}}}

\renewcommand{\phi}{\varphi}
\newcommand{\upchi}{{\raise.35ex\hbox{\ensuremath{\chi}}}}

\newcommand{\Alg}{\operatorname{Alg}}
\newcommand{\Ad}{\operatorname{Ad}}

\newcommand{\Aut}{\operatorname{Aut}}

\newcommand{\dg}{\operatorname{diag}}

\newcommand{\id}{{\operatorname{id}}}

\newcommand{\Rad}{\operatorname{Rad}}
\newcommand{\spn}{\operatorname{span}}

\newcommand{\rt}{\operatorname{rt}}

\newcommand\diag{\mathop{\rm diag}}

\newcommand{\lt}{\operatorname{lt}}
\newcommand{\ca}{\mathrm{C}^*}
\newcommand{\cenv}{\mathrm{C}^*_{\text{env}}}
\newcommand{\cmax}{\mathrm{C}^*_{\text{max}}}

\newcommand{\sot}{\textsc{sot}}

\newenvironment{sbmatrix}{\left[\begin{smallmatrix}}
{\end{smallmatrix}\right]}

\newcommand\cpf{\rtimes_{\alpha}\, {\mathcal{G}}}
\newcommand\cps{\rtimes_{\alpha}\, {\Sigma}}

\begin{document}

\title[Takai duality]{Crossed Products of operator algebras: applications of Takai Duality}

\author[E.G. Katsoulis]{Elias~G.~Katsoulis}
\address {Department of Mathematics
\\East Carolina University\\ Greenville, NC 27858\\USA}
\email{katsoulise@ecu.edu}

\author[C. Ramsey]{Christopher~Ramsey}
\address {Department of Mathematics
\\University of Manitoba\\ Winnipeg, MB \\Canada}
\email{Christopher.Ramsey@umanitoba.ca}


\maketitle

%%%%%%%%%%%%%%%%
\begin{abstract}
Let $(\G, \Sigma)$ be an ordered abelian group with Haar measure $\mu$, let  $(\A, \G, \alpha)$ be a dynamical system and let $\A\rtimes_{\alpha} \Sigma $ be the associated semicrossed product. Using Takai duality we establish a stable isomorphism
\[
\A\rtimes_{\alpha} \Sigma \sim_{s} \big(\A \otimes \K(\G, \Sigma, \mu)\big)\rtimes_{\alpha\otimes \Ad \rho} \G,
\]
where $\K(\G, \Sigma, \mu)$ denotes the compact operators in the CSL algebra $\Alg\L(\G, \Sigma, \mu)$ and $\rho$ denotes the right regular representation of $\G$.
We also show that there exists a complete lattice isomorphism between the $\hat{\alpha}$-invariant ideals of $\A\rtimes_{\alpha} \Sigma$ and the $(\alpha\otimes \Ad \rho)$-invariant ideals of $\A \otimes \K(\G, \Sigma, \mu)$. 

Using Takai duality we also continue our study of the Radical for the crossed product of an operator algebra and we solve open problems stemming from the earlier work of the authors. Among others we show that the crossed product of a radical operator algebra by a compact abelian group is a radical operator algebra. We also show that the permanence of semisimplicity fails for crossed products by $\bbR$. A final section of the paper is devoted to the study of radically tight dynamical systems, i.e., dynamical systems $(\A, \G, \alpha)$ for which the identity $\Rad(\A\cpf)=(\Rad\A) \cpf$ persists. A broad class of such dynamical systems is identified.

\end{abstract}

\section{introduction} 

In this paper we continue our study of crossed products of (mostly non-selfadjoint) operator algebras begun with our monograph \cite{KR}. The main objectives of the paper is to strengthen ties and establish new connections between our theory of crossed products and the well-established theory of semicrossed products. Using these connections we broaden our understanding for various topics of investigation in both theories, including semisimplicity and the structure of invariant ideals by the dual action.

Starting with work of Arveson in the sixties \cite{Arv3, ArvJ}, the study of semicrossed products by $\bbZ^+$ (or other discrete semigroups) has been a central topic of investigation in operator algebra theory and has produced a steady stream of important results \cite{DFK,  DavKatCr, DavKatAn, DavKatMem,  KakKatJFA1, Pet, Pet2}. Quite recently, the authors of the present paper started studying crossed products of operator algebras \cite{KR}, in connection with several problems in both selfadjoint and non-selfadjoint operator algebra theory. Throughout \cite{KR}, it was pointed out that certain semicrossed product algebras of the form $\A \rtimes_{\alpha} \bbZ^+$ were actually isomorphic to crossed products of operator algebras by the action of the full group $\bbZ$; see for instance \cite[Problem 6]{KR}.  This raises the natural question as to how and to what extend these two classes of operator algebras are related to each other. The non-selfadjoint Takai duality of \cite[Theorem 4.4]{KR} certainly guarantees that any operator algebra $\A$ admitting an action $\alpha$ by $\bbT$ is stably isomorphic to a crossed product of the form $\B\rtimes_{\beta}\bbZ$, with $\B=\A \rtimes_{\alpha} \bbT$ and $\beta=\hat{\alpha}$. However, identifying a familiar representative of $\A \rtimes_{\alpha} \bbT$ is a formidable task and it is an even more difficult task to identify how the dual action manifests on that representative of $\A \rtimes_{\alpha} \bbT$. Nevertheless, here we are able to establish that up to stable isomorphism any semicrossed product $\A \rtimes_{\alpha} \bbZ^+$ can be thought of as a crossed product of a very concrete operator algebra by a very explicit action of $\bbZ$. (Actually our results work in much greater generality and include crossed products by other natural semigroups, including $\bbR^+$.) Indeed, in Theorem~\ref{stable isom} we show that for any ordered abelian group $(\G, \Sigma)$ (see below for definitions) and dynamical system  $(\A, \G, \alpha)$ there exists a stable isomorphism
\[
\A\rtimes_{\alpha} \Sigma \sim_{s} \big(\A \otimes \K(\G, \Sigma, \mu)\big)\rtimes_{\alpha\otimes \Ad \rho} \G.
\]
Here $\rho$ denotes the right regular representation of $\G$ and $\K(\G, \Sigma, \mu)$ the compact operators in the CSL algebra $\Alg\L(\G, \Sigma, \mu)$.

Using Theorem~\ref{stable isom} we now revisit the concept of semisimplicity for the crossed product of an operator algebra and answer a problem that was left open in \cite{KR}. In \cite[Problem 4]{KR} we asked whether the semisimplicity of $\A$ implies the semisimplicity of $\A \rtimes_{\alpha} \bbR$ and vice versa. This problem was motivated by the fact that if $\A$ is a semisimple operator algebra, then $\A \rtimes_{\alpha} \G$ is semisimple, for any \textit{discrete} dynamical system $(\A, \alpha, \G)$. As it turns out, Example~\ref{Problem 5} shows that the answer is negative not only for crossed products by $\bbR$ but also for any non-discrete ordered abelian group $\G$. 

Section~\ref{invariant section} contains several applications. First in Theorem~\ref{Peters param} we give a complete lattice parametrization of all ideals of the semicrossed product $\A \cps$ which are invariant by the dual action. We do this in an indirect way using Theorem~\ref{stable isom} and the dynamical system $( \A\otimes  \K ( \G , \Sigma, \mu ), \G, \alpha \otimes\Ad \rho)$ appearing in that theorem. This works for \textit{any} abelian group admitting an ordered structure, not just discrete ones, and sheds new light on an old result of Peters \cite[III.5. Theorem]{Pet2}. A key step in the proof of Theorem~\ref{Peters param} is a complete characterization of all ideals of the crossed product $\A \cps$ which are invariant by the dual action; this generalizes an old result of Gootman and Lazar \cite{GL1}. In Section~\ref{invariant section} we also (partially) solve another problem from \cite{KR} by showing that the diagonal of a crossed product is what it oughts to be, i.e., $\dg\big(\A\cpf \big)=\dg(\A) \cpf$, provided that $\G$ is compact and abelian, (see Theorem~\ref{compact radical}.) In Section~\ref{invariant section} we also show that the crossed product of a radical operator algebra by a compact group is again a radical algebra. This leads naturally to the concept of radical tightness which is explored in the last section.

Section~\ref{Section tight} is occupied with a detailed study of the Jacobson radical for a crossed product. A dynamical system $(\A, \G, \alpha)$ is called radically tight if $\Rad(\A \rtimes_\alpha \G) = (\Rad \A) \rtimes_\alpha \G$. This condition need not happen as was shown in \cite[Example 6.8]{KR} and will be further shown in the semisimplicity discussion of Section~\ref{Section stable}. In Theorem~\ref{tight} we establish that every dynamical system consisting of a strongly maximal TUHF algebra and a discrete abelian group is radically tight. The proof of Theorem~\ref{tight} is technical and among others, it requires the resolution of an old problem stemming from the work of Hudson.  In \cite{Hudson} Hudson studied various notions of a radical radical for a TAF algebra with the linkless ideal being contained in all of them and Jacobson Radical being the largest. In Theorem~\ref{TUHFRadical} we show that for any strongly maximal TUHF algebra the linkless ideal and the Jacobson Radical coincide. This validates Hudson's intuition who remarked in \cite[pg. 228]{Hudson} that for all familiar examples of TAF algebras the two ideals seemed to coincide and yet he was not able to provide any specific results in that direction.

\section{Stable isomorphisms of semicrossed products} \label{Section stable}

Let $\G$ be a \textit{second countable} locally compact abelian group with Haar measure $\mu$ (all groups appearing in this paper are abelian).  Let $\Sigma \subseteq \G$ be a \textit{cone} in $\G$ satisfying
\begin{itemize} 
\item[(i)] $\Sigma\cap \Sigma^{-1} = \{1\}$
\item[(ii)] $\Sigma  \cdot \Sigma \subseteq \Sigma$
\item[(iii)] $\Sigma$ equals the closure of its interior. 
\end{itemize}
We say that the pair $(\G ,\Sigma)$ forms an  \textit{ordered abelian group}. Using the ordered group structure $(\G, \Sigma)$ and a $\sigma$-finite Borel measure $m$ on $\G$, we define a commutative subspace lattice $\L(\G, \Sigma, m)$ as follows. 

Consider the partial order $\leq$ in $\G$ to mean $s \leq t$ only when $s^{-1}t \in \Sigma$. A Borel set $ E\subseteq \G$ is said to be increasing if for any $ s \in E$ and $t \in \G$, $s \leq t$ implies $t \in E$. This is a standard order in the terminology of \cite{Arv1}. 

Let $P_E$ be the projection on $L^2(\G, m)$ given by multiplication by the characteristic function of $E$. Then
\[
\L(\G, \Sigma, m) \equiv \{ P_E\mid E \subseteq \G \mbox{ increasing}\}
\]
is a commutative subspace lattice, which up to unitary equivalence depends only on the equivalence class of the measure $m$. 

\begin{example}  \label{Z-ordered}
$\G = \bbZ$, $\Sigma=\bbZ^+$, $m$ the counting measure. 
By identifying closed subspaces with their orthogonal projections, it is easy to see that $\L(\G, \Sigma, m)$ is unitarily equivalent to the \textit{discrete $\bbZ$-ordered nest}
\[
\big\{  P_n \mid P_n = [\{e_i,  n \leq i <\infty\}], n \in \bbZ \big\},
\]
where $\{e_n\}_{n\in \bbZ}$ is the standard orthonormal basis of $l^2(\bbZ)$.
\end{example} 

\begin{example} \label{Voltera nest} $\G = \bbR$, $\Sigma=\bbR^+$, $m$ the Lebesgue measure. 
In this case $\L(\G, \Sigma, m)$ is unitarily equivalent to the \textit{Volterra nest}
\[
\{L^2([t , +\infty))\mid t \in \bbR\}.
\]
\end{example} 

\begin{example}
$\G = \bbR^2$, $\Sigma=\bbR^+\times \bbR^+$, $m$ the Lebesgue area measure. 
In this case $\L(\G, \Sigma, m)$ is the commutative subspace lattice generated by the nests
\[
\{L^2([t , +\infty)\times \bbR) \mid t \in \bbR\} \mbox{ and } \{ L^2(\bbR \times [s , +\infty) )\mid s \in \bbR\}.
\]
\end{example}

 In order to utilize results on finite measures, we find it convenient to consider the finite measure $m \equiv hd\mu$, where $h \in L^1(\G)\cap L^2(\G)$ is a suitable positive continuous function so that $m\equiv \mu$. The invariance of the Haar measure under translations implies that the function 
\[
\G \ni s \longmapsto m(s E), \quad E \subseteq \G\  \mbox{ increasing}
\] is continuous. Hence the measure $m$ is $\Sigma$-continuous  in the terminology of \cite[Definition 6.1]{Arv2}  and so $\mu(\vartheta E) = m(\vartheta E)=0$, for all increasing sets $E$ (and their complements) by \cite[Proposition 6.2]{Arv2}. 

If $\L$ is a commutative subspace lattice, then $\Alg \L$ will denote its CSL algebra, i.e., the collection of all operators leaving invariant (the range of) each element of $\L$. The study of CSL algebras was initiated by Arveson in his seminal paper \cite{Arv1}. One of the important topics of investigation in that theory revolves around the density of the compact operators. It is known that the $\sot$-density (or $w^*$-density) of the span of rank-one operators in $\Alg \L$ (rank-one density) is equivalent to $\L$ being completely distributive~\cite{LL}. The complete distributivity of $\L$ is also equivalent to the density of the Hilbert-Schmidt operators in $\Alg \L$ \cite{HLM}. In both cases one can choose an approximate unit consisting of contractions~\cite{DP}. Deciding when the compact operators are dense in a CSL algebra remains to date an open problem. 

If $(\G ,\Sigma)$ is an ordered group and $m$ a $\sigma$-finite measure on $\G$, then the collection of all compact operators in $\Alg  \L(\G, \Sigma, m)$ will be denoted as $\K(\G, \Sigma, m)$. In the case where $m$ is finite and $\Sigma$-continuous, Laurie ~\cite{Laurie} has shown that the lattice $ \L(\G, \Sigma, m) $ is completely distributive and so the span of rank-one operators is dense in $\K(\G, \Sigma, m)$ (and strongly dense in $\Alg  \L(\G, \Sigma, \mu)$). By considering the measure $ m \equiv hd\mu \sim \mu$ discussed earlier, the same is true for  $\L(\G, \Sigma, \mu)$ and $\K(\G, \Sigma, \mu)$.

\begin{example} \label{example alg} Let $\G = \bbZ$, $\Sigma=\bbZ^+$ and $m$ be the counting measure (Example \ref{Z-ordered}). In that case $\Alg \L(\G, \Sigma, m)$ consists of all (bounded) lower triangular infinite matrices 
$A = [\{a_{i j}\}]_{i, j \in \bbZ}$ with entries in $\bbC$. It is easy to see that $$A = [\{a_{i j}\}] \in \K(\G, \Sigma, m)$$ provided that the finite truncations $[\{a_{i j}\}]_{i, j =-n}^{n}$, $n \in \bbZ^+$, approximate $A$ in norm.

In the case where $\G = \bbR$, $\Sigma=\bbR^+$ and $m$ is the Lebesgue measure (Example~\ref{Voltera nest}), there is no matricial description for either $\Alg \L(\G, \Sigma, m)$ or $\K (\G, \Sigma, m)$. Notice however that the Hilbert-Schmidt operators in\break $\Alg \L(\G, \Sigma, m)$ can be described as the integral operators on $L^2(\bbR, m)$ whose kernel functions are square integrable and supported on 
\[
\{(s, t) \mid s\geq t, \, \, s, t \in\bbR\}.
\]
Our next result elaborates on that theme. 
\end{example}

\begin{lemma} \label{integral op}
The collection of all Hilbert-Schmidt integral operators of the form
\[
Tf(s)  = \int K(s, t)f(t ) d \mu(t), \,\, s \in \G, f \in L^2(\G),
\]
with $ K \in C_c(\G \times \G)$ supported on the graph 
\[
(\G, \leq)\equiv \{(s, t)\mid s\geq t, \, \, s, t \in \\G\},
\]
forms a dense subset of $\K(\G, \Sigma, \mu)$.
\end{lemma}

\begin{proof}
It is enough to show that the span of the rank-one operators  in $\Alg  \L(\G, \Sigma, m)$ satisfying the requirements of the lemma forms a dense set. Then by unitary equivalence and the special form of $h$, where $m = h d \mu$, we also have the result for $\K(\G, \Sigma, \mu)$. 

Let $e\otimes f$ denote the rank-one operator on $L^2(\G)$ defined as $(e\otimes f)(h)= \langle h\mid e\rangle f $, where $e,f, h \in L^2(\G)$. By \cite[Lemma 23.3]{DavNest} if $e \otimes f \in \Alg  \L(\G, \Sigma, m)$, then there exist increasing Borel sets $E, F$ so that $e =e \chi_{E^c} $ and $f = f \chi_{F}$; furthermore if $e',f' \in L^2(\G, m)$ satisfy  $e' = e' \chi_{E^c}$ and $f' = f' \chi_{F}$, then $e' \otimes f' \in \Alg  \L(\G, \Sigma, m)$ as well. Since $m(\vartheta E^c)= m(\vartheta F) =0$, we conclude that the rank-one operators of the form $e\otimes f \in \Alg  \L(\G, \Sigma, m)$, $e,f \in C_c(\G)$, are dense in the set of all rank-one operators. By \cite[Theorem 22.5]{DavNest} such rank-one operators have continuous kernel supported on $(\G, \leq)$ and so the closure of all integral operators described in the statement of the Lemma contains all rank-one operators in $\Alg  \L(\G, \Sigma, m)$. By \cite[Theorem 23.7 (ii)]{DavNest}, the rank-one operators in $\Alg  \L(\G, \Sigma, m)$ are dense in $\K(\G, \Sigma, \mu)$ and the conclusion follows.
\end{proof}

Let $(\A , \G,\Sigma, \alpha)$ be an \textit{ordered dynamical system}, i.e., a dynamical system $(\A, \G, \alpha)$ with $(\G , \Sigma )$ an ordered abelian group. In addition to the crossed products $\A\cpf\subseteq \cenv(\A) \cpf$, the presence of the ordered structure allows us to introduce the \textit{semicrossed product} $\A \cps$ as the norm closed subalgebra of $\cenv(\A) \cpf$ generated by $ C_c(\Sigma, \A)$, i.e., all continuous $\A$-valued functions defined on $\G$ with compact support contained in $\Sigma$

\begin{remark} \label{alternative} 
(i) We make the following convention throughout this section. If $X \subseteq \G$ (resp. $X \subseteq \G \times \G)$, then $C_c(X)$ will denote the continuous functions on $G$ (resp. $\G \times \G$) with compact support contained in $X$. A similar convention applies to the $\A$-valued functions comprising $C_c(X, \A)$.

Note that in the case where $X$ is a open subset, $C_c(X)$ can be identified naturally with the continuous functions with compact support which are defined only on $X$.

(ii) Alternatively one can define $\A \cps$ as the norm closed subalgebra of $\cenv(\A) \cpf$ generated by $C_c(\mathring{\Sigma}, \A)$. 

Indeed, if $f \in C_c(\Sigma)$, then its restriction on $\mathring{\Sigma}$ can be approximated in the $L^1$-norm by elements of $C_c(\mathring{\Sigma})$. Since $\mu(\vartheta\Sigma)=0$, the same $L^1$-approximation applies to $f $ itself. From this we conclude that elementary tensors in $C_c(\Sigma, \A)$ can be approximated in the operator norm by elements of $C_c(\mathring{\Sigma}, \A)$ and the conclusion follows.
\end{remark}

In the case where $\G$ is discrete, our definition of $\A \cps$ coincides with the one commonly appearing in the literature \cite{DavKatAn, DavKatCr, Pet, Pet2}. The same is true in the non-discrete case as well but that requires some explanation. McAsey and Muhly \cite[pg 129]{MM1} define the semicrossed product as the norm closed subalgebra of $\cenv(\A) \cpf$ generated by the set of functions in $L^1(\G, \A)$ which are supported on $\Sigma$. Note however that $\mu(\vartheta \Sigma)= 0$ and so we need only to consider functions in $L^1(\mathring{\Sigma}, \A)$. Such functions can be approximated in the operator norm by $C_c(\mathring{\Sigma}, \A)$ and so by the remark above, both definitions describe the same object.

In the sequel we utilize various actions of $\G$ (and $\Sigma$) on certain $\ca$-algebras. For instance $\G$ acts on $C_0(\G)$ be left and right translation, denoted as $\lt$ and $\rt$ respectively, i.e., $\lt_s(f)(t) = f(s^{-1}t)$ and $\rt_s(f)(t)=f(ts)$, $s, t \in G$. The trivial action of $\G$ on any $\ca$-algebra is denoted as $\id$.

\begin{proposition}Let $(\G, \Sigma)$ be an ordered group. Then there exists an equivariant isomorphism $\psi$ from
\[
\big(C_0(\G)\rtimes_{\lt}\Sigma, \G, \rt\otimes \id\big) \mbox{ onto } \big(\K(\G , \Sigma, \mu), \G, \Ad\rho\big)
\]
defined on $f \in C_c(\Sigma \times \G)\subseteq C_0(\G)\rtimes_{\lt}\Sigma$ and $ g \in C_c(\G)$by 
\[
\psi(f)(g)(s) = \int f(r, s) g(r^{-1}s)d\mu(r).
\]
\end{proposition}
\begin{proof}
In the case $\Sigma = \G$ the statement of the Proposition is the content of the Stone-Von Neumann Theorem, which is actually verified using the same equivariant map $\psi$. What we need to additionally verify here is that $$\psi\big( C_c(\Sigma \times \G) \big) \subseteq \K(\G , \Sigma, \mu)$$ as a dense subset.

First notice that if $f \in C_c(\Sigma \times \G)$, then 
\[
\psi(f)(g)(s) =   \int f(r, s) g(r^{-1}s)d\mu(r) = \int K(s, r)g(r ) d \mu(r),
\]
with $K(s, r )\equiv f(r^{-1}s,s)$. If $K(s, r)\neq 0$, then clearly $r^{-1}s \in \Sigma$ and so $s \geq r$. Hence the support of $K$ is contained in the graph of $(\G , \leq)$ and so  $\psi\big( C_c(\Sigma \times \G) \big) \subseteq \K(\G , \Sigma, \mu)$.

To prove density, let $K \in C_c(\G \times \G)$ supported on $(\G, \leq)$. Then 
\[
f_K(r, s) = K(s, r^{-1} s ), \, r, s \in \G
\]
defines an element of  $C_c(\G \times \G)$ that satisfies 
\[
\psi(f_K)(g)(s) = \int K(s, r) g(r)d\mu(r).
\]
Note that if $r \notin \Sigma$, then $s \ngeq r^{-1}s$ and so since $K$ supported on $(\G, \leq)$, we have
\[
f_K(r, s) = K(s, r^{-1} s )=0,
\]
i.e., $f_K \in C_c(\Sigma \times \G)$. Hence the range $\psi$ contains all integral operators appearing in the statement of Lemma~\ref{integral op} and density follows.

For the equivariance note that 
\begin{align*}
\rho(v)\psi(f)(g)(s) &= \psi(f)(g)(sv) \\
			    &= \int f(r, sv) g(r^{-1}sv)d\mu(r) \\
			    &=\int (\rt\otimes\id)_v (f)(r, s) \rho(v)(g)g(r^{-1}s)d\mu(r)\\
			    &=\psi\big((\rt\otimes\id)_v (f)\big)\rho(v)(g)(s)
\end{align*}
and so $\rho(v)\psi(f)\rho(v^{-1})=\psi\big((\rt\otimes\id)_v (f)\big)$, as desired.
 \end{proof}
 
Under the additional assumption that $\Sigma$ generates $\G$ as a group, one can deduce the surjectivity of $\psi$ from \cite[Theorem 5.1]{MM2}. Note however that the proof of \cite[Theorem 5.1]{MM2} relies heavily on the theory of $\ca$-dynamical systems and makes good use of the Arveson spectrum. Our proof of course is of a very different flavor as it avoids $\ca$-dynamics and instead uses the theory of CSL algebras.
 
 \begin{corollary} \label{newStonecor}
 Let $(\A , \G,\Sigma, \alpha)$ be an ordered dynamical system. Then there exists a completely isometric equivariant isomorphism $\Psi \equiv \psi\otimes \id$ from 
 \[\big( C_0(\G, \A)\rtimes_{\lt\otimes \id}\Sigma, \G, (\rt\otimes \alpha) \otimes \id\big) \mbox{ onto }
\big(  \A \otimes \K ( \G , \Sigma, \mu ), \alpha \otimes \Ad \rho\big).\]
 \end{corollary}
 
 Given a dynamical system $(\A, \G, \alpha)$, we have a homomorphism
 \[
 \hat{\alpha}\colon \hat{\G} \longrightarrow \Aut (\A \cpf)
 \]
 which is given by $\hat{\alpha}(f)(s) = \overline{\gamma(s)}f(s)$, for $\gamma \in \hat{\G}$ and $f \in C_c(\G, \A)$. The dynamical system $(\A \cpf, \hat{\G}, \hat{\alpha})$ is called the \textit{dual system} and $\hat{\alpha}$ the \textit{dual action}. The well-known Takai duality~\cite{KR, Takai} explains what happen when we take the crossed product of the dual system. Proposition\ref{Takai duality} below explains what happens when we take the crossed product of a \textit{semicrossed product} by the dual action. But first we need the following
 
 \begin{lemma} \label{phi id}
 Let $(\G,\Sigma)$ be an ordered abelian group. Let $(\A , \G,\alpha)$ and $(\B , \G,\beta)$ be dynamical systems and let $\phi : \A \rightarrow \B$ be an equivariant completely isometric isomorphism. Then the integrated form of $\phi$,
 \[
 \phi\rtimes \id: \A\cpf \longrightarrow \B\rtimes_{\beta} \G,
 \]
 where $\phi\rtimes\id(f)(s)\equiv \phi(f(s))$, $f \in C_c(\G, \A)$, $s \in \G$, establishes a completely isometric isomorphism whose restriction on $\A \rtimes_{\alpha}\Sigma$ maps onto $\B\rtimes_{\beta} \Sigma$.
 \end{lemma}
 
 \begin{proof}
 Consider the completely isometric extension $\tilde{\phi}: \cenv(\A) \rightarrow \cenv (\\B)$ of $\phi$ and apply \cite[Corollary 2.48]{Will} to obtain a $*$-isomorphism 
 \[
 \tilde{\phi}\rtimes \id: \cenv(\A) \cpf \longrightarrow  \cenv(\B) \rtimes_{\beta} \G,
 \]
 where $\tilde{\phi}\rtimes\id(f)(s)\equiv \tilde{\phi}(f(s))$, $f \in C_c(\G, \cenv(\A))$, $s \in \G$. The fact that the restriction of $\tilde{\phi}\rtimes\id$ on $\A \rtimes \G$ produces the desired map $\phi \rtimes \id$ with the desired surjectivity follows by examining elementary tensors.
 
 Indeed, if $a \in \A$ and $f \in C_c(\mathring{\Sigma})$ then $\phi \rtimes \id (af) =\phi(a)f$. Since elementary tensors of the form $af$, $a \in \A$, $f \in C_c(\mathring{\Sigma})$, span a dense subset of $C_c(\mathring{\Sigma}, \A)$, we obtain that
 \[
 \phi \rtimes \id \big(C_c(\mathring{\Sigma}, \A)\big) = C_c(\mathring{\Sigma}, \B).
 \]
 Therefore Remark~\ref{alternative}(ii) implies $\phi \rtimes \id (\A \rtimes_{\alpha}\Sigma)= B\rtimes_{\beta} \Sigma$, as desired
 \end{proof}

\begin{proposition}[Takai duality for semicrossed products] \label{Takai duality}
Let $(\A , \G,\Sigma, \alpha)$ be an ordered dynamical system. Then there exists complete isomorphism 
\[
\Phi \colon \
\big( \A \rtimes_{\alpha} \Sigma\big)\rtimes_{\hat{\alpha}}\hat{\G} \longrightarrow \A \otimes \K ( \G , \Sigma, \mu ),
\]
which is equivariant for the double dual action $\hat{\hat{\alpha}} \colon \G \rightarrow \Aut \big((\A \rtimes_{\alpha} \Sigma \big)\rtimes_{\hat{\alpha}}\hat{\G}\big) $ and the action $\alpha \otimes \Ad \rho \colon\G \rightarrow \Aut \big( \A \otimes  \K ( \G , \Sigma, \mu )\big)$.
\end{proposition}

\begin{proof} We appeal to the proof of the Takai Duality Theorem as it appears in \cite{KR} (and also in \cite{Will}). There we construct an equivariant completely isometric isomorphism
$\widetilde{\Phi}$ from
\[
\big( ( \A \cpf)\rtimes_{\hat{\alpha}}\hat{\G}, \G, \hat{\hat{\alpha}}\big) \mbox{ onto } \big( C_0(\G, \A)\rtimes_{\lt \otimes \id}\G , \G , (\rt\otimes \alpha) \otimes \id \big)
\]
which maps an $F \in C_c(\hat{\G}\times \G, \A)$ to the element in $C_c\big(\G, C_0(\G, \A)\big)$ given by
\[
\widetilde{\Phi}(F)(s, r)= \int_{\hat{\G}}\alpha_r^{-1}\big(F(\gamma, s)\big)\overline{\gamma(s^{-1}r)}d\gamma.
\]
Actually  $\widetilde{\Phi} =  \widetilde{\Phi}_3\circ  \widetilde{\Phi}_2 \circ  \widetilde{\Phi}_1$, where 
\begin{align*} 
\widetilde{\Phi}_1 &: \big( \A \cpf\big)\rtimes_{\hat{\alpha}}\hat{\G} \longrightarrow \big( \A\rtimes_{\id} \hat{\G}\big)\rtimes_{\hat{\id}^{-1}\otimes \alpha } \G \\
\widetilde{\Phi}_2 &: \big( \A\rtimes_{\id} \hat{\G}\big)\rtimes_{\hat{\id}^{-1}\otimes \alpha } \G \longrightarrow C_0(\G, \A) \rtimes_{\lt \otimes \alpha} \G \\
\widetilde{\Phi}_3 &: C_0(\G, \A) \rtimes_{\lt \otimes \alpha} \G
 \longrightarrow C_0(\G, \A) \rtimes_{\lt \otimes \id} \G, 
\end{align*}
are defined by
\begin{align*}
\widetilde{\Phi}_1(F)(s, \gamma) &= \gamma(s)F(\gamma, s) \\
\widetilde{\Phi}_2(F)(s, r)&= \int_{\hat{\G}} F(s, \gamma)\overline{\gamma(r)}d\gamma\\
\widetilde{\Phi}_3(F) (s,r)&= \alpha_{r^{-1}}\big(F(s, r)\big)
\end{align*} 
From the first of the above formulas one can easily verify that 
\[
\widetilde{\Phi}_1\big( (\A \rtimes_{\alpha}\Sigma)\rtimes_{\hat{\alpha}}\hat{\G} \big) =\big( \A\rtimes_{\id} \hat{\G}\big)\rtimes_{\hat{\id}^{-1}\otimes \alpha } \Sigma. \]
It is also true that 
\[\widetilde{\Phi}_2\big(( \A\rtimes_{\id} \hat{\G})\rtimes_{\hat{\id}^{-1}\otimes \alpha } \Sigma\big) = C_0(\G, \A) \rtimes_{\lt \otimes \alpha} \Sigma \]
and 
\[\widetilde{\Phi}_3\big(C_0(\G, \A) \rtimes_{\lt \otimes \alpha} \Sigma\big)
 = C_0(\G, \A) \rtimes_{\lt \otimes \id} \Sigma
\]
but this follows from a different reasoning. To verify the first of these formulas, note that $\widetilde{\Phi}_2$ is the integrated form of a $\G$-equivariant isomorphism between $\A \rtimes_{\id}\hat{\G}$ and $C_0(\G, \A)$ (see \cite[Lemma 7.3]{Will}) and so Lemma~\ref{phi id} applies to show surjectivity. The second formula follows similarly (see \cite[Lemma7.4]{Will}).

Therefore by restricting $\widetilde{\Phi}$ we obtain a completely isometric isomorphism
$\Phi$ from
\[
\big( ( \A \rtimes_{\alpha} \Sigma)\rtimes_{\hat{\alpha}}\hat{\G}, \G, \hat{\hat{\alpha}}\big)
\]
onto 
\[\big( C_0(\G, \A)\rtimes_{\lt \otimes \id}\Sigma , \G , (\rt\otimes \alpha) \otimes \id \big).
\]
Now $ \hat{\hat{\alpha}}$ is given on $C_c(\hat{\G}\times \Sigma) \subseteq (\A \rtimes_{\alpha}\Sigma)\rtimes_{\hat{\alpha}}\hat{\G} $ by
\[
 \hat{\hat{\alpha}}_t(F)(\gamma, s)= \overline{\gamma(t)}F(\gamma, s)
 \]
and so 
\begin{align*}
\Phi \big( \hat{\hat{\alpha}}_t(F)\big)(s, r)&= \int_{\hat{\G}} \alpha_{r}^{-1}\big( \hat{\hat{\alpha}}_t(F)(\gamma, s) \big) \overline{\gamma(s^{-1}r)} d \gamma \\
&=\alpha_t\Big(  \int_{\hat{\G}} \alpha_{rt}^{-1}\big( \hat{\hat{\alpha}}_t(F)(\gamma, s) \big) \overline{\gamma(s^{-1}rt)} d \gamma   \Big) \\
&=\alpha_t\big( \Phi(F)(s,rt) \big).
\end{align*}
This shows that $\Phi\circ  \hat{\hat{\alpha}}_t = (\rt\otimes \alpha)_t\otimes\id \circ \Phi$, i.e., $\Phi$ is equivariant. The conclusion now follows from Corollary~\ref{newStonecor}.
\end{proof}

We have arrived at the main result of this section. It makes a very precise connection between the well-established theory of semicrossed products and the theory of crossed products of arbitrary operator algebras, which is more recent. 

\begin{theorem} \label{stable isom}  If $(\A , \G,\Sigma, \alpha)$ is an ordered dynamical system, then we have a stable isomorphism
\[
\A\rtimes_{\alpha} \Sigma \sim_{s} \big(\A \otimes  \K(\G, \Sigma, \mu)\big)\rtimes_{\alpha\otimes \Ad \rho} \G. 
\]
\end{theorem}

\begin{proof}
From Proposition~\ref{Takai duality} we obtain
\[
\Big(\big( \A \rtimes_{\alpha} \Sigma\big)\rtimes_{\hat{\alpha}}\hat{\G} \Big)\rtimes_{ \hat{\hat{\alpha}}} \G \simeq \big(\A \otimes \K ( \G , \Sigma, \mu ) \big)\rtimes_{\alpha\otimes \Ad \rho} \G.
\]
By applying non-selfadjoint Takai duality \cite[Theorem 4.4]{KR} on the dynamical system $( \A \rtimes_{\alpha} \Sigma, \hat{\alpha}, \hat{\G})$ we obtain
\[\Big(\big( \A \rtimes_{\alpha} \Sigma\big)\rtimes_{\hat{\alpha}}\hat{\G} \Big)\rtimes_{ \hat{\hat{\alpha}}} \G \simeq ( \A \rtimes_{\alpha} \Sigma)\otimes \K\big( L^2(\hat{\G})\big). 
\]
By ``equating" the right sides of the above equivalences we have 
\[
( \A \rtimes_{\alpha} \Sigma)\otimes \K\big( L^2(\hat{\G})\big) \simeq  \big(\A \otimes \K ( \G , \Sigma, \mu ) \big)\rtimes_{\alpha\otimes \Ad \rho} \G,
\]
which establishes the desired stable isomorphism.
\end{proof}

\begin{remark} Comparing the stable isomorphism of the usual Takai duality 
\[
\big((\A \rtimes_{\alpha}\Sigma)\rtimes_{\hat{\alpha}} \hat{\G}\big)\rtimes_{\hat{\hat{\alpha}}} \G\sim_s \A \rtimes_{\alpha}\Sigma
\]
with the stable isomorphism
\[
 \big(\A \otimes  \K(\G, \Sigma, \mu)\big)\rtimes_{\alpha\otimes \Ad \rho} \G \sim_s\A\rtimes_{\alpha} \Sigma
\] of Theorem~\ref{stable isom}, one may be tempted to conclude that $(\A \rtimes_{\alpha}\Sigma)\rtimes_{\hat{\alpha}} \hat{\G}$ is isomorphic to $\A \otimes  \K(\G, \Sigma, \mu)$ and that $\hat{\hat{\alpha}}$ is equivariant to $\alpha\otimes \Ad \rho$. We want to emphasize that as natural as this isomorphism might seem, its existence is not known to us. We were only able to verify that $(\A \rtimes_{\alpha}\Sigma)\rtimes_{\hat{\alpha}} \hat{\G}$ is stably isomorphic to $\A \otimes  \K(\G, \Sigma, \mu)$, which was sufficient for our purposes. \end{remark}

We want to understand better the structure of the algebra $\A \otimes  \K(\G, \Sigma, \mu)$ appearing in the stable isomorphism of Theorem~\ref{stable isom}. If $\G$ is not discrete, then Corollary~\ref{radicalalg} below says something definitive. But first we need the following two results.

\begin{proposition} \label{non atomic}
Let $(\G, \Sigma)$ be an ordered group and assume that $\G$ is not discrete, i.e., $\G$ has no isolated points. Then the lattice $\L(\G, \Sigma, \mu)$ has no atoms, i.e., no minimal intervals.
\end{proposition}

\begin{proof}
Since $\G$ is not discrete $\mu(g)= 0$, for any $ g \in \G$. We work with the unitarily equivalent lattice $\L(\G, \Sigma, m)$. Let $E \subseteq F\subseteq \G$ increasing so that $m( E \backslash F)>0$. Because $m(\vartheta E)=m(\vartheta F)=0$, we have $m(S)>0$, where $S\equiv \mathring F\backslash \overline{E}$.

Choose $x \in $ satisfying $m(I_x\cap U)>0$, where $I_x\equiv x \Sigma$ is the smallest increasing set containing $x$. (See \cite[Lemma 2]{Laurie}.) If $m(I_x\cap U) < m(U)$, then we see that 
\[
P_E \lneqq P_{E\cup I_x}\lneqq P_F
\]
and so the arbitrarily chosen interval $P_E \lneqq P_F$ is not minimal, as desired. Therefore we may assume that $m(I_x\cap U) = m(U)$.

We will construct a decreasing sequence $\{x_n\}_{n=1}^{\infty}$  converging to $x$ so that the sequence $m(I_{x_{n}}\cap U)\}_{n=1}^{\infty}$ increases to $m(I_x \cap U)$ without eventally being constant. Notice that if such a sequence $\{x_n\}_{n=1}^{\infty}$ is constructed, then any term $x_n$ with $m(I_{x_{n}}\cap U)>0$ satisfies 
\[
P_E \lneqq P_{E\cup I_{x_n}} \lneqq P_F
\]
and so once again the interval $P_E \lneqq P_F$ is not minimal.

To construct the sequence  $\{x_n\}_{n=1}^{\infty}$, we modify the ideas of \cite[Proposition 6.2]{Arv2}. 

Let $\{U_n\}_{n=1}^{\infty}$ be a decreasing sequence of open subsets of $U$ with $\cap_{n}U_n=\{x\}$.
Let $V_1 = U_1 \cap x\mathring \Sigma$ and let $x_1\in V_1$. Now note that $U_2 \cap x_1 \mathring \Sigma$ is an open neighborhood of $x$ and so $V_2\equiv  U_2 \cap  x\mathring \Sigma \cap x_1 \mathring \Sigma$ is a non-empty open set; choose $x_2 \in V_2$. Continuing in that manner, we construct non-empty open sets 
\[
V_n \equiv U_n \cap x\mathring \Sigma \cap\big(\cap_{i=1}^{n-1} x_i \mathring \Sigma\big) \subseteq U
\]
and $x_n \in V_n$, for all $n \in \bbN$. As in \cite[Proposition 6.2]{Arv2}, one sees that $\{x_n\}_{n=1}^{\infty}$ is decreasing to $x$. Furthermore 
\[
V_n \cap I_{x_{n-1}}\subseteq \{x_{n-1}\}, \mbox{ for all } n\in \bbN.
\]
(Indeed, $V_n \subseteq x_{n-1} \mathring \Sigma$ and so $v \leq x_{n-1}$, for all $v \in V_n$.) Therefore, 
\[
m(V_n \cap I_{x_{n-1}})= 0 , \mbox{ for all }n \in \bbN.
\]
Since $m$ does not annihilates open sets and $V_n  \subseteq U$, we conclude that $m(I_{x_{n}}\cap U) < m(U) = m(I_x\cap U)$. On the other hand, since $\{x_n\}_{n=1}^{\infty}$ is decreasing to $x$, $I_x = \overline{\cup_n I_{x_n}}$. Since the boundary of increasing sets has measure $0$, we have
\[
m\big(I_x \backslash \cup_{n=1}^{\infty}I_{x_n}\big)  = 0 
\]
and so 
\[
m\big((I_x \cap U )\backslash \cup_{n=1}^{\infty}(I_{x_n}\cap U) \big) = m\big((I_x \backslash \cup_{n=1}^{\infty}I_{x_n})\cap U\big) = 0  
\] 
as desired. This completes the proof.
\end{proof}

\begin{proposition} \label{diag int}
Let $\L$ be a non-atomic commutative subspace lattice and let $K \in \Alg \L$ be a compact operator. Then any operator of the form $A \otimes K$, $A \in B(\H)$, is quasinilpotent.
\end{proposition}

\begin{proof}
Since $\L$ has no atoms, the diagonal integral of $K$ with respect to $\L$ equals $0$ by \cite[Theorem 5.4]{Kats}. Hence $K$ is quasinilpotent 
and the the conclusion follows.
\end{proof}

\begin{corollary} \label{radicalalg}
Let $(\G, \Sigma)$ be an ordered group and assume that $\G$ is not discrete. If $\A$ is any operator algebra, then $\A\otimes  \K ( \G , \Sigma, \mu ) $ is a radical operator algebra, i.e., a Banach algebra consisting of quasinilpotent operators.
\end{corollary}

\begin{proof}
The proof follows from the previous result and Proposition~\ref{non atomic}.
\end{proof}

If $\G$ is discrete then $\A\otimes  \K ( \G , \Sigma, \mu ) $ may not be a radical algebra anymore but nevertheless it is often the case that its Radical is non-trivial. For instance, in the case where $(\G, \Sigma)$ is either $(\bbZ, \bbZ^+)$ or $(\bbQ, \bbQ^+)$, one has 
\[
\Rad \big(\A\otimes  \K ( \G , \Sigma, \mu ) \big) = \A\otimes  \K_0 ( \G , \Sigma, \mu ) ,
\]
where $  \K_0 ( \G , \Sigma, \mu )$ denotes the compact operators in $ \K_0 ( \G , \Sigma, \mu )$ with zero diagonal. (This follows easily from either \cite[Corollary 6.9]{DavNest} or from \cite[Theorem 5.4]{Kats} as in Proposition~\ref{diag int}.) In particular, $\A\otimes  \K ( \G , \Sigma, \mu )$ is not semisimple.

The previous results have important consequences for the study of semicrossed products. For one thing, the stable isomorphism of Theorem~\ref{stable isom} implies that the lattice of the ideals of $\A \rtimes_{\alpha} \Sigma$ has the same (complete)  structure as the lattice of ideals for $\big( \A \otimes \K ( \G , \Sigma, \mu ) \big)\rtimes_{\alpha\otimes \Ad \rho} \G$. So in principle, question regarding the ideal structure of  $\A \rtimes_{\alpha} \Sigma$, e.g., semisimplicity, are delegated to corresponding questions about the ideals of $\big( \A \otimes \K ( \G , \Sigma, \mu ) \big)\rtimes_{\alpha\otimes \Ad \rho} \G$.  More importantly however, Theorem~\ref{stable isom}  leads for the first time to a ``meaningful' parametrization of the ideals of $\A \rtimes_{\alpha} \Sigma$ which are invariant by the dual action; this will be explained in Section~\ref{invariant section}. For the moment, we take the opposite route: we import results from the theory of semicrossed products in order to understand  and answer problems in crossed product theory. 

We begin with an example that answers a problem from the authors \cite[Problem 4]{KR}.

\begin{example} \label{Problem 5}
\textit{Let $\G$ be a non-discrete abelian group that admits an ordered group structure, e.g., $\G=\bbR$ and  $\Sigma =\bbR^+$. Then there exists a dynamical system $(\B, \beta, \G)$, with $\B$ being a radical Banach algebra so that $\B \rtimes_{\beta} \G$ is semisimple.}
\end{example}

\noindent Indeed by considering $\A=\bbC$, $\alpha =\id$ and $\G$ as above, the stable isomorphism of Theorem~\ref{stable isom} obtains 
\[
\bbC\rtimes_{\id} \Sigma \sim_{s} \K(\G,\Sigma, \mu)\rtimes_{\Ad \rho} \G .
\]
Now notice that $\bbC\rtimes_{\id} \Sigma\subseteq \bbC\rtimes_{\id} \G \simeq C_0(\hat{\G})$. Therefore $\bbC\rtimes_{\id} \Sigma$ is semisimple and so by Lemma~\ref{compact ideal}, $\K(\G, \Sigma, \mu)\rtimes_{\Ad \rho} \G$ is also semisimple.  On the other hand, Corollary \ref{radicalalg} shows that $\K(\G, \Sigma, \mu)$ is a radical Banach algebra. By considering $\B\equiv \K(\G, \Sigma, \mu)$ and $\beta=\Ad\rho$, the conclusion follows.

\vspace{.1in}

Example~\ref{Problem 5} implies in particular the existence of a dynamical system $(\A, \bbR, \alpha)$ so that $\A \rtimes_{\alpha} \bbR$ is semisimple but $\A$ fails to be such. This answers in the negative one of the questions in~\cite[Problem 4]{KR}. Since $\bbR$ is self-dual with respect to Pontryagin duality, another application of Takai duality gives a dynamical system $(\A, \bbR, \alpha)$ so that $\A$ is semisimple but $\A \rtimes_{\alpha} \bbR$ is not. Indeed, if $(\B, \G, \beta )$ is as in Example~\ref{Problem 5}, with $\G = \bbR$, then choose $\A \equiv \B \rtimes_{\beta} \bbR$ and $\alpha \equiv \hat{\beta}$ the dual action. Thus both questions in \cite[Problem 4]{KR} have a negative answer.

Even though many non-discrete groups admit an ordered group structure, there is a notable class that does not: the \textit{compact} abelian groups. Not only does Example~\ref{Problem 5} not apply to compact abelian groups but as we shall see in Corollary \ref{compact Problem 4}, nothing like Example~\ref{Problem 5} can happen there.

As we mentioned before, one of the motivating results in \cite{KR} is Theorem 6.2 which asserts the permanence of semisimplicity under crossed products by discrete groups, i.e., if $\G$ is discrete and $\A$ is semisimple, the $\A\cpf$ is also semisimple. Even though this is not a difficult result to establish, verifying that its converse fails required considerable effort~\cite[Example 6.8 and Theorem 6.9]{KR}. Given the developments of the present paper, additional counterexamples to the converse of \cite[Theorem 6.2]{KR} are not so difficult to come by now. (Note however that in contrast to the examples below, all counterexamples in \cite{KR} are unital algebras.)

\begin{example}
Let $(\X, \phi)$ be a dynamical system with $\X$ a compact  Hausdorff space and $\phi$ a minimal homeomorphism of $\X$. Then the operator algebra $\big(C(\X) \otimes  \K(\bbZ, \bbZ^+, \mu)\big)\rtimes_{\phi\otimes \Ad\rho} \bbZ$ is semisimple but $C(\X) \otimes  \K(\bbZ, \bbZ^+, \mu)$ is not semisimple.
\end{example}

\noindent Indeed, as we discussed in the proof of Corollary~\ref{radicalalg}, the algebra $C(\X) \otimes  \K(\bbZ, \bbZ^+, \mu)$ is not semisimple. On the other hand a result of Muhly \cite{M} shows that $C(\X)\rtimes_{\phi} \bbZ^+$ is semisimple and so $\big(C(\X) \otimes  \K(\bbZ, \bbZ^+, \mu)\big)\rtimes_{\phi\otimes \Ad\rho} \bbZ$ is semisimple by Theorem~\ref{stable isom} and the soon to be seen Lemma~\ref{compact ideal}. 

As it turns out, the previous example can be strengthened and it leads to the following characterization of semisimplicity for a natural class of crossed product algebras.

Let $X$ be a locally compact metric space and $\phi:X \rightarrow X$ be a homeomorphism. A point $x \in X$ is said to be recurrent if there exists an increasing sequence $\{ n_k\}_{k = 1}^{\infty}$ so that $\lim_{k} \phi^{(n_k)}(x) = x$.

\begin{theorem} Let $X$ be a locally compact metric space and $\phi:X \rightarrow X$ be a homeomorphism. Then the following are equivalent
\begin{itemize}
\item[(i)] $\big(C(X) \otimes \K(\bbZ, \bbZ^+, \mu)\big)\rtimes_{\phi\otimes \Ad\rho} \bbZ$ is semisimple
\item[(ii)] the recurrent points of $(X, \phi)$ are dense in $X$.
\end{itemize}
\end{theorem}

\begin{proof} 
The semisimplicity of $\big( C(X) \otimes \K(\bbZ, \bbZ^+, \mu)\big)\rtimes_{\phi\otimes \Ad\rho} \bbZ$ is equivalent to the semisimplicity of $C(X)\rtimes_{\phi}\bbZ^+$ because of the stable isomorphism. In \cite{DKM}, the semisimplicity of  $C(X)\rtimes_{\phi}\bbZ^+$ is shown to be equivalent to the density of the recurrent points and the conclusion follows.
\end{proof}

\section{Invariant ideals by the dual action and the Radical} \label{invariant section}

Let $(\X, \phi)$ be a dynamical system consisting of a homeomorphism $\phi$ of a locally compact space $\X$. If $\{ X_n\}_{n=1}^{\infty}$ is a sequence of closed subsets of $\X$ satisfying 
\begin{equation} \label{star}
X_{n+1}\cup \phi(X_{n+1})\subseteq X_n, \quad n=0, 1, \dots,
\end{equation}
then it is easy to see that all $a \in C(\X)\rtimes_{\phi}\bbZ^+$ whose Fourier expansion $a \sim \sum_{n=0}^{\infty} f_nU^n$ satisfies $f_n(X_n)=\{0\}$, for all $n=0, 1, \dots$, form a closed ideal of the semicrossed product $C(\X)\rtimes_{\phi}\bbZ^+$ which is invariant by the dual action of $\bbT$, i.e., gauge invariant. In \cite{Pet2} Peters shows that conversely, any gauge invariant ideal of $C(\X)\rtimes_{\phi}\bbZ^+$ arises that way and that the association between sequences $\{ X_n\}_{n=1}^{\infty}$  satisfying (\ref{star}) and gauge invariant ideals of $C(\X)\rtimes_{\phi}\bbZ^+$ is injective. Furthermore in the case where $\phi$ acts freely on $\X$, this scheme describes \textit{all} closed ideals of $C(\X)\rtimes_{\phi}\bbZ^+$.\footnote{Because of this result, a metatheorem would state that Peters' scheme is the broadest parametrization of ideals available for semicrossed products by $\bbZ^+$.} 

In this section we wish to extend  Peters' parametrization to more general semicrossed products in some meaningful way. Even though it is not difficult to see how to do this with discrete semigroups, moving to more general semigroups becomes tricky and literally impossible to verify whenever a Fourier series development is not available. As it turns out, the key object for doing this is the dynamical system $( \A\otimes  \K ( \G , \Sigma, \mu ), \G, \alpha \otimes\Ad \rho)$ of Theorem~\ref{stable isom}. Using this dynamical system we give a complete parametrization of the ideals invariant by the dual action; this appears in Theorem~\ref{Peters param} and it is one of the main results of this section. In this section we also (partially) solve another problem from \cite{KR} by showing that the diagonal of a crossed product is what it ought to be, i.e., $\dg\big(\A\cpf \big)=\dg(\A) \cpf$, provided that $\G$ is compact and abelian. Finally in Corollary~\ref{compact Problem 4} we see that the crossed product of a radical operator algebra by a compact abelian group remains a radical algebra. This strongly contrasts the situation with other non-compact groups (compare with Example~\ref{Problem 5}) and shows that the study of the Radical for a crossed product is rich in theorems as well, not just counterexamples.

In our next result, Lemma~\ref{compact ideal}, we gather some elementary facts regarding tensor products with the compact operators. Parts of Lemma~\ref{compact ideal} have already been used in the previous section.

\begin{lemma} \label{compact ideal}
Let $\A$ be an operator algebra and let $\K(\H)$ denote the compact operators acting on a Hilbert space $\H$. 
\begin{itemize}
\item[(i)] If $\J \subseteq \A \otimes \K(\H)$ is a closed ideal, then there exists a closed ideal $\I \subseteq \A$ so that $\J = \I \otimes \K(\H)$. 

\noindent In particular 
\begin{equation} \label{radical perm}
\Rad \big(  \A \otimes \K(\H)\big) = \Rad \A \otimes \K(\H)
\end{equation}
and therefore stable isomorphisms preserve semisimplicity.
\item[(ii)] We have $$\dg\big( \A \otimes \K(\H)\big) = \dg\A \otimes \K(\H).$$
\end{itemize}
\end{lemma}

\begin{proof}
Let $\{ \xi_i \}_{i \in \bbI}$ be an orthonormal basis for $\H$ and let $e_{i j}$ be the rank one operator mapping $\xi_j$ to $\xi_i$, $ i, j \in \bbI$. Assume that $\A$ acts on some Hilbert space $\H'$. Then $\A \otimes \K(\H)$ is generated as an operator algebra by all elementary tensors of the form $a \otimes e_{ ij} \in B(\H' \otimes \H)$, $ a \in \A$, $ i, j \in \bbI$.

To prove (i), fix an $i_0 \in \bbI$. It is easy to see that
\[
\I \equiv \{ s \in \A\mid s\otimes e_{i_0 i_0}= (I \otimes e_{i_0 i_0} )S (I \otimes e_{i_0 i_0} ), \mbox{ for some } S \in \J\}
\]
is the desired ideal.

We now verify that $\Rad \A \otimes \K(\H) \subseteq \Rad \big(  \A \otimes \K(\H)\big)$. Let $x \in \Rad \A$ and fix an $i_0 \in \bbI$. It is easy to see that given any $A \in \A \otimes\K(\H)$, there exists $a \in \A$ so that 
\[
a\otimes e_{i_0 i_0}  = (I \otimes e_{i_0 i_0}) A (I \otimes e_{i_0 i_0} )
\]
and so 
\begin{align*}
(A(x\otimes  e_{i_0 i_0}))^n &= A(x\otimes  e_{i_0 i_0})  \Big( (I \otimes e_{i_0 i_0}) A (I \otimes e_{i_0 i_0} ) (x\otimes e_{i_0 i_0})\Big)^{n-1} \\
    &=A(x\otimes  e_{i_0 i_0})  \big( (ax)^{n-1}\otimes e_{i_0 i_0}\big)
\end{align*}
for all $n \in \bbN$.
Hence 
\[
\begin{split}
\lim_n \|(A(x\otimes  e_{i_0 i_0}))^n\|^{1/n}&\leq \lim_n \| A(x\otimes  e_{i_0 i_0}) \|^{1/n} \limsup_n  \|(ax)^{n-1}\|^{1/n} \\
                                               &=\limsup_n  \|(ax)^{n}\|^{1/n} =0
                                               \end{split}
\]
because $x \in \Rad \A$. Hence $ x \otimes e_{i_0 i_0} \in \Rad\big( \A\otimes \K( \H)\big)$. From this it follows that $x \otimes e_{ij} \in \Rad\big( \A\otimes \K( \H)\big)$, for all $x \in \Rad \A$ and $i, j \in \bbI$. Hence $\Rad \A \otimes \K(\H) \subseteq \Rad \big(  \A \otimes \K(\H)\big)$. Since the reverse inclusion is easy, the proof of (i) is complete.

To prove (ii), assume first that $\A$ is unital  and let $a \in \dg\big( \A \otimes \K(\H)\big)$. Clearly $1\otimes e_{i j} \in \dg\big( \A \otimes \K(\H)\big)$ for all $i, j\in \bbI$ and so 
\[
a_{i, j}\otimes e_{i j} \equiv (1\otimes e_{i i} )a(1\otimes e_{j j} ) \in \dg\big( \A \otimes \K(\H)\big).
\]
Since $\{1\otimes e_{ii}\}_{i \in \bbI}$ is an approximate unit for $ \dg\big( \A \otimes \K(\H)\big)$, we are to show that $a_{i,j} \in \dg \A$, for all $i,j \in \bbI$.

Consider families $\{ b_{st}^{(n)}\}$, $s, t \in \bbI$, $n \in \bbN$, in $\A$ so that 
\[
a = \lim_n \sum _{s,t} ({b_{st}^{(n)}})^*\otimes e_{st}.
\]
Then, 
\begin{align*}
a_{ij}\otimes e_{ij}&= (1\otimes e_{i i} )a(1\otimes e_{j j} ) \\
			&= \lim_n (1\otimes e_{i i} )\big(\sum _{s, t} ({b_{st}^{(n)}})^*\otimes e_{st}\big)(1\otimes e_{j j} ) \\
			&= \lim_n \, ({b_{ij}^{(n)}})^*\otimes e_{ij}
			\end{align*}
and so $a_{i, j} =  \lim_n \, ({b_{ij}^{(n)}})^* \in \A^*$, as desired. 

If $\A$ is not unital, then one replaces $1\in \A$ with an approximate unit for the $\ca$-algebra $\diag \A$ consisting of selfadjoint operators and then repeats the same arguments as above by taking limits.
\end{proof}

Let $( \C, \G, \alpha)$ be a $\ca$-dynamical system with $\G$ abelian and let 
\[
\Phi \colon \
\big( \C \cpf\big)\rtimes_{\hat{\alpha}}\hat{\G} \longrightarrow \C \otimes \K \big( L^2 (\G)\big),
\]
be the isomorphism which implements the Takai duality. If $\A \subseteq \C$ is a closed subalgebra of $\C$ which is left invariant by the action $\alpha \colon \G \rightarrow \Aut \C$, then the proof of Theorem~\ref{Takai duality} shows that
\[
\Phi\left( ( \A \cpf\big){\rtimes}_{\hat{\alpha}}\hat{\G}\right)=\A \otimes \K \big( L^2 (\G)\big),
\]
and the restriction of $\Phi$ on $( \A \cpf\big)\rtimes_{\hat{\alpha}}\hat{\G}$ establishes the desired equivariance. Here the crossed product $\A \cpf$ is indeed isomorphic to the subalgebra of $\C \cpf $ generated by $C_c(\G , \A)$ and similarly for $( \A \cpf\big){\rtimes}_{\hat{\alpha}}\hat{\G}$.

Our next theorem generalizes and strengthens  a result of Gootman and Lazar  \cite[Corollary 2.2]{GL1}. 

\begin{theorem} \label{newGL}
Let $(\A , \G, \alpha)$ be a dynamical system with $\G$ a locally compact abelian group. Then the association 
\begin{equation} \label{lattice1}
\A \supseteq \J \longmapsto \J \cpf \subseteq \A \cpf
\end{equation}
establishes a complete lattice isomorphism between the $\alpha$-invariant ideals of $\A$ and the $\hat{\alpha}$-invariant ideals of $\A \cpf$.
\end{theorem}

\begin{proof} Let 
\begin{equation} \label{newGLiso}
\Phi \colon \
\big( \cenv(\A) \cpf\big)\rtimes_{\hat{\alpha}}\hat{\G} \longrightarrow \cenv(\A) \otimes \K \big( L^2 (\G)\big),
\end{equation}
be the isomorphism guaranteed by Takai duality, which is equivariant for the double dual action $\hat{\hat{\alpha}}$ of $\G$ on $(\cenv(\A) \cpf\big)\rtimes_{\hat{\alpha}}\hat{\G} $ and the action $\alpha \otimes \Ad \rho \colon\G \rightarrow \Aut \big( \cenv(A) \otimes  \K\big( L^2(\G)\big)\big)$. 

Let $\J \subseteq \A \cpf$ be an ideal which is invariant by the dual action $\hat{\alpha}$. Then the ideal 
\[
\J \rtimes_{\hat{\alpha}}\hat{\G} \subseteq \big( \A \cpf\big)\rtimes_{\hat{\alpha}}\hat{\G} 
\]
is mapped by $\Phi$ to an ideal of $\A \otimes \K \big( L^2 (\G)\big)$. By Lemma~\ref{compact ideal}, there exists an ideal $\I \subseteq \A$ so that $\Phi(\J \rtimes_{\hat{\alpha}}\hat{\G})  = \I \otimes \K(\H) $. Notice that since $\J \rtimes_{\hat{\alpha}}\hat{G}$ is invariant by the double dual action $\hat{\hat{\alpha}}$, the ideal $\I \otimes \K(\H)  = \Phi(\J \rtimes_{\hat{\alpha}}\G)$ is invariant by the action  $\alpha \otimes \Ad \rho$ (by equivariance). Hence $\I$ is invariant by $\alpha$.

Finally, both $(\I\cpf) \rtimes_{\hat{\alpha}}\hat{\G}$ and $ \J \rtimes_{\hat{\alpha}}\hat{\G}$ are mapped by $\Phi$ onto $\I \otimes \K(\H) $ and so 
\[
(\I\cpf) \rtimes_{\hat{\alpha}}\hat{\G} = \J \rtimes_{\hat{\alpha}}\hat{\G}
\]
as subsets of $\big(\cenv(\A)\cpf\big) \rtimes_{\hat{\alpha}}\hat{\G} $. Hence,
\[
\big((\I\cpf) \rtimes_{\hat{\alpha}}\hat{\G} \big)\rtimes_{\hat{\hat{\alpha}}}\G= \big(\J \rtimes_{\hat{\alpha}}\hat{\G}\big)\rtimes_{\hat{\hat{\alpha}}}\G
\]
as subsets of $\big((\cenv(\A)\cpf) \rtimes_{\hat{\alpha}}\hat{\G} \big)\rtimes_{\hat{\hat{\alpha}}}\G$ and so 
\[
\hat{\Phi} (\I \cpf) = \hat{\Phi}(\J),
\]
where 
\begin{equation} \label{newGLiso2}
\hat{\Phi} \colon \
\big( (\cenv(\A) \cpf )\rtimes_{\hat{\alpha}}\hat{\G}\big)\rtimes_{\hat{\hat{\alpha}}}\G \longrightarrow \big(\cenv(\A) \cpf \big)\otimes \K \big( L^2 (\hat{\G})\big),
\end{equation}
is the other map guaranteed by Takai Duality for the dynamical system  $( \cenv(\A)\cpf, \hat{\G}, \hat{\alpha})$. But this last equation simply means
\[
\big(\I\cpf\big)\otimes\K\big(L^2(\hat{\G}) \big)= \J \otimes\K\big(L^2(\hat{\G})\big)
\]
and so $\J =\I \cpf$. Thus the association in (\ref{lattice1}) is surjective. It is easily seen to be injective as well.

The fact that the association in (\ref{lattice1}) respects lattice operations is verified along similar lines; we do this with the lattice operation of ``intersection" and we leave the ``closure of the union" for the reader.

Assume that $\{ \I_k\}$ is a collection of $\alpha$-invariant ideals of $\A$. We are to prove that 
\[
\cap_k(\I_k \cpf)= (\cap_k \I_k)\cpf.
\]
We use again Takai duality by invoking both maps $\Phi$ and $\hat{\Phi}$ of (\ref{newGLiso}) and (\ref{newGLiso2}) respectively. Indeed notice that 
\begin{align*}
\big(\cap_k(\I_k \cpf)\big) \rtimes_{\hat{\alpha}} \hat{\G}& \subseteq \cap_k\big( (\I_k \cpf) \rtimes_{\hat{\alpha}} \hat{\G} \big)\\
				&=\cap_k\Phi^{-1}\big(\I_k \otimes\K\big( L^2(\G)\big)\big ) \\
				&= \Phi^{-1}\Big( \cap_k \big(\I_k \otimes\K\big( L^2(\G)\big)\big ) \Big)\\
				&=\Phi^{-1} \Big(\big(\cap_k \I_k\big ) \otimes\K\big( L^2(\G)\big)\Big)\\
				&=\big( (\cap_k \I_k)\cpf\big) \rtimes_{\hat{\alpha}}\hat{\G}.
\end{align*}
Since both $\big(\cap_k(\I_k \cpf)\big) \rtimes_{\hat{\alpha}} \hat{\G}$ and $\big( (\cap_k \I_k)\cpf\big) \rtimes_{\hat{\alpha}}\hat{\G}$ are $\hat{\hat{\alpha}}$-invariant, we have 
\[
\Big(\big(\cap_k(\I_k \cpf)\big) \rtimes_{\hat{\alpha}} \hat{\G}\Big) \rtimes_{\hat{\hat{\alpha}}} \G\subseteq \Big(\big( (\cap_k \I_k)\cpf\big) \rtimes_{\hat{\alpha}}\hat{\G}\Big) \rtimes_{\hat{\hat{\alpha}}}\G
\]
and so 
\[
\hat{\Phi}\big(\cap_k(\I_k \cpf)\big) \subseteq \hat{\Phi}\big((\cap_k \I_k)\cpf\big).
\]
This implies that $\cap_k(\I_k \cpf) \subseteq(\cap_k \I_k)\cpf$ and since the reverse inclusion is obvious, the conclusion follows.
\end{proof}

\begin{theorem} \label{Peters param}Let $(\G, \Sigma)$ be an ordered abelian group with Haar measure $\mu$ and let  $(\A, \G, \alpha)$ be a dynamical system. Then there exists a complete lattice isomorphism between the $\hat{\alpha}$-invariant ideals of $\A\rtimes_{\alpha} \Sigma$ and the $(\alpha\otimes \Ad \rho)$-invariant ideals of $\A \otimes \K(\G, \Sigma, \mu)$. 
\end{theorem}

\begin{proof} Apply Theorem~\ref{newGL} to the dynamical system $(\A \rtimes_{\alpha}\Sigma, \hat{\G}, \hat{\alpha})$ to obtain a complete lattice isomorphism between the $\hat{\alpha}$-invariant ideals of $\A\rtimes_{\alpha} \Sigma$ and the $\hat{\hat{\alpha}}$-invariant ideals of $\big(\A\rtimes_{\alpha} \Sigma\big) \rtimes_{\hat{\alpha}}\hat{\G}$. 

Consider now the map 
\[
\Phi \colon \
\big( \A \rtimes_{\alpha} \Sigma\big)\rtimes_{\hat{\alpha}}\hat{\G} \longrightarrow \A \otimes \K ( \G , \Sigma, \mu ),
\]
of Proposition~\ref{Takai duality} which is equivariant for the double dual action $\hat{\hat{\alpha}} \colon \G \rightarrow \Aut \big((\A \rtimes_{\alpha} \Sigma \big)\rtimes_{\hat{\alpha}}\hat{\G}\big) $ and the action $\alpha \otimes \Ad \rho \colon\G \rightarrow \Aut \big( \A \otimes  \K ( \G , \Sigma, \mu )\big)$. By equivariance, this map $\Phi$ establishes a complete lattice isomorphism between the $\hat{\hat{\alpha}}$-invariant ideals of $\big(\A\rtimes_{\alpha} \Sigma\big) \rtimes_{\hat{\alpha}}\hat{\G}$ and the $(\alpha\otimes \Ad \rho)$-invariant ideals of $\A \otimes \K(\G, \Sigma, \mu)$. This completes the proof.
\end{proof}

\begin{remark}
It is instructive to see how Theorem~\ref{Peters param} recovers Peters' scheme for parametrizing the gauge invariant ideals of a semicrossed product by $\bbZ$.

Let  $(\A, \bbZ, \alpha)$ be a dynamical system and let $\I\subseteq \A\otimes \K(\bbZ, \bbZ^+, \mu)$ be an $(\alpha \otimes \Ad \rho )$-invariant ideal. We need to uniquely associate to $\I$ a sequence $\{ \I_n\}_{n =1}^{\infty}$ of ideals of $\A$ which satisfy
\begin{equation} \label{bigstar}
\I_n\subseteq \I_{n+1}\cap\alpha(\I_{n+1}), \,\, n=0, 1, \dots .
\end{equation}
(In the case where $\A=C(\X)$ and $\alpha(f) = f \circ \phi^{-1}$, $f \in C(\X)$, this is easily seen to lead to a sequence of closed sets satisfying (\ref{star}).)

The algebra $\A\otimes \K(\bbZ, \bbZ^+, \mu)$ is isomorphic to the collection of all lower triangular infinite matrices 
$A = [\{a_{i j}\}]_{i, j \in \bbZ}$ with entries in $\A$, so that the finite truncations $[\{a_{i j}\}]_{i, j =-n}^{n}$, $n \in \bbZ^+$, approximate $A$ in norm. (Compare with Example~\ref{example alg}.) Given an $(\alpha \otimes \Ad \rho )$-invariant ideal $\I\subseteq \A\otimes \K(\bbZ, \bbZ^+, \mu)$ we define $\I_{i j}\subseteq \A$ to be the collection of all $(i, j)$-entries of $\I$; it is easily seen that each $\I_{i, j}$ is an ideal of $\A$. Also let $\hat{\I}_{i,j}$ be the subspace consisting of all elements of $\I$ with all entries apart from the $(i, j)$-entry being zero. Clearly the $(i, j)$-entry of $\hat{\I}_{i,j}$ varies over $\I_{i j}$. 

Since $(\alpha\otimes\Ad\rho)_s(\I) = \I$, $s \in \G$, we obtain
\[
\alpha_s (\I_{n ,0}) = \I_{-s +n ,-n}, \, \, n=0, 1, \dots.
\]
Thus $\I_{n ,0}$ completely determines all other ideals appearing on the $n$th lower diagonal. By multiplying $\hat{\I}_{n,0}$ from the left with $e_{n+1, n}\in\A\otimes \K(\bbZ, \bbZ^+, \mu)$, i.e., the matrix which is $1$ at the $(n+1, n)$-entry and $0$ elsewhere, we obtain 
\[
\I_{n,0}\subseteq \I_{n+1,0}, \,\,  n=0, 1, \dots
\] and similarly, by multiplying $\hat{\I}_{n+1,1}$from the right with $e_{1,0}$ we obtain
\[
\alpha_{-1}(\I_{n,0})=\I_{n+1,1}\subseteq \I_{n+1,0}, ,\,  n=0, 1, \dots.
\]
Putting together these two inclusions we obtain
\[
\I_{n, 0}\subseteq \I_{n+1,0}\cap\alpha(\I_{n+1,0}), \,\, n=0, 1, \dots .
\]
By setting $\I_n\equiv \I_{n,0}$ we obtain (\ref{bigstar}). Clearly the association between $(\alpha \otimes \Ad \rho )$-invariant ideals of $\A\otimes \K(\bbZ, \bbZ^+, \mu)$ and sequences $\{ \I_n\}_{n =1}^{\infty}$ satisfying (\ref{bigstar}) is injective. It is easily seen to be surjective as well.

It is not difficult to extend (\ref{bigstar}) to arbitrary discrete groups and modify Peters' scheme to hold there, even without appealing to our Theorem~\ref{Peters param}. (The proof requires an argument with a Fejer-type kernel.) Nevertheless nothing like that seems to work when one moves beyond discrete groups; there is neither an obvious substitute for (\ref{bigstar}) nor a Fejer-type kernel to work with. The use of $\A \otimes \K(\G, \Sigma, \mu)$ seems to be the only appropriate choice.
\end{remark}

In Theorem~\ref{newGL} we gave a complete description of all $\alpha$-invariant ideals for the crossed product $\A\cpf$. Automorphism invariant ideals are a central theme in non-selfadjoint operator algebra theory with most important being the (Jacobson) Radical. Such ideals have been studied extensively~\cite{Don, DKM, M, Rin, O1, O2}.  

The following is immediate from Theorem~\ref{newGL}.

\begin{corollary} \label{Radideal}
Let $\fG = (\A , \G, \alpha)$ be a dynamical system with $\G$ a locally compact abelian group. Then there exists an $\alpha$-invariant ideal $\J_{\fG} \subseteq \A$ so that 
\[
\Rad\big(\A \cpf\big)= \J_{\fG}\cpf  .
\]
\end{corollary}

One would hope that $\J_{\fG} = \Rad \A$ but as we saw in the previous section, this fails in many ways. The relation between $\J_{\fG}$ and $ \Rad \A$ is one of the main themes of this paper. If $\G$ is discrete one can easily see that $\J_{\fG} \subseteq  \Rad \A$, with the inclusion being proper in certain cases. In the case of a compact group, the situation reverses.

\begin{theorem} \label{compact rradical}
Let $\fG = (\A , \G, \alpha)$ be a dynamical system with $\G$ a compact abelian group. Then there exists an $\alpha$-invariant ideal $  \J_{\fG} \supseteq \Rad \A$ so that 
\[
\Rad\big(\A \cpf\big)= \J_{\fG}\cpf  .
\]
\end{theorem}

\begin{proof}
In light of Corollary \ref{Radideal}, we have that $\Rad\big(\A \cpf\big)= \J_{\fG} \cpf$ for some ideal $\J_{\fG} \subseteq \A$; we only need to verify that $\Rad \A \subseteq \J_{\fG}$.

Since $\hat{\G}$ is discrete, there exists ideal $\I \subseteq \Rad\big( \A \cpf\big)$ so that 
\[
\Rad \big( (\A \cpf )\rtimes_{\hat{\alpha}}\hat{\G}\big) = \I \rtimes_{\hat{\alpha}}\hat{\G}.
\]
Consider the isomorphism
\[
\Phi \colon \
\big( \cenv(\A) \cpf\big)\rtimes_{\hat{\alpha}}\hat{\G} \longrightarrow \cenv(\A) \otimes \K \big( L^2 (\G)\big)
\]
guaranteed by Takai duality and notice that
\begin{align*}
\Phi\big(\Rad \big( (\A \cpf )\rtimes_{\hat{\alpha}}\hat{\G}\big)\big)&= \Rad \big( \Phi\big((\A \cpf )\rtimes_{\hat{\alpha}}\hat{\G}\big)\big)\\
			&=\Rad\big( \A \otimes  \K ( L^2 (\G))\big)\\
			&=\Rad\A \otimes  \K ( L^2 (\G))
\end{align*}
by Lemma~\ref{compact ideal} (i). Hence 
\[
\Phi ( \I \rtimes_{\hat{\alpha}}\hat{\G} )= \Rad\A \otimes  \K ( L^2 (\G)).
\]
Also 
\[
\Phi\big((\Rad \A \cpf )\rtimes_{\hat{\alpha}}\hat{\G}\big) = \Rad\A \otimes  \K ( L^2 (\G))
\]
and since $\Phi$ is an isomorphism,
\[
 \I \rtimes_{\hat{\alpha}}\hat{\G} = (\Rad \A \cpf )\rtimes_{\hat{\alpha}}\hat{\G}.
 \]
 Since $\hat{\G}$ is discrete, we obtain $$\I = \Rad \A \cpf.$$
However, $\I \subseteq \Rad\big( \A \cpf\big)$ and so 
\[
\Rad \A \cpf \subseteq  \Rad\big( \A \cpf\big) = \J_{\fG}\cpf.
\]
By another application of Takai duality $\Rad \A \subseteq \J_{\fG}$, as desired.
\end{proof}

The following contrasts strongly Example~\ref{Problem 5}.

\begin{corollary} \label{compact Problem 4}
Let $\fG = (\A , \G, \alpha)$ be a dynamical system with $\G$ a compact abelian group. If $\A$ is a radical Banach algebra, then $\A \cpf$ is also a radical Banach algebra.
\end{corollary}

The technique of Theorem~\ref{compact rradical} also allows us to address an issue left open in \cite[Problem 5]{KR}. Indeed in \cite[Proposition 5.11]{KR} we calculated the diagonal of a crossed product $\A \cpf$ in the case where $\G$ is an amenable and discrete group. Here we consider another important class of groups.

\begin{theorem} \label{compact radical}
Let $\A$ be a unital operator algebra and $\alpha \colon \G \rightarrow \Aut \A$ be the continuous action of a compact abelian group. Then 
\[
\dg\big(\A\cpf \big)=\dg(\A) \cpf .
\]
\end{theorem}

\begin{proof}
Clearly, 
\[
\dg\A \cpf \subseteq \dg\big(\A\cpf \big).
\]
By way of contradiction assume that the above inclusion is proper. Both $\dg\A \cpf $ and $\dg\big(\A\cpf \big)$ are invariant by the dual action $\hat{\alpha}$ when considered as subalgebras of $ \cmax(\A)  \cpf $. Since they are distinct, 
\begin{equation} \label{comp ineq}
\big(\dg\A \cpf\big)\rtimes_{\hat{\alpha}}\hat{\G} \neq  \dg\big(\A\cpf \big)\rtimes_{\hat{\alpha}}\hat{\G}.
\end{equation}
Now we examine separately each side of (\ref{comp ineq}) under the action of the isomorphism
\[
\Phi \colon \
\big( \cmax(\A) \cpf\big)\rtimes_{\hat{\alpha}}\hat{\G} \longrightarrow \cmax(\A) \otimes \K \big( L^2 (\G)\big)
\]
guaranteed by Takai duality.  

For the left side, we apply Takai duality for the $\ca$-dynamical system $(\dg\A, \G, \alpha)$ and we obtain
\begin{align*}
\Phi\Big(\big(\dg\A \cpf\big)\rtimes_{\hat{\alpha}}\hat{\G} \Big)&= \dg\A \otimes  K(L^2(\G))  \\
											&=  \dg \big( \A \otimes K(L^2(\G)) \big)
\end{align*}
by Lemma~\ref{compact ideal}.

For the right side of (\ref{comp ineq}) notice that since $\hat{\G}$ is discrete, \cite[Proposition 5.11]{KR} implies that 
\[
\dg\big(\A\cpf \big)\rtimes_{\hat{\alpha}} \hat{\G} = \dg\Big(\big(\A\cpf \big)\rtimes_{\hat{\alpha}}\hat{\G}\Big)
\]
and so
\begin{align*}
\Phi\Big(\dg\big(\A\cpf \big)\rtimes_{\hat{\alpha}} \hat{\G}\Big)	&= \Phi\Big( \dg\big( (\A\cpf )\rtimes_{\hat{\alpha}}\hat{\G}\big)\Big) \\
						&= \dg \Phi\big( (\A\cpf )\rtimes_{\hat{\alpha}}\hat{\G})\big) \\
								&= \dg \big( \A \otimes K(L^2(\G)) \big).
\end{align*}
By comparing sides we have 
\[
\Phi\Big(\big(\dg\A \cpf\big)\rtimes_{\hat{\alpha}}\hat{\G} \Big) = \Phi\Big(\dg\big(\A\cpf \big)\rtimes_{\hat{\alpha}} \hat{\G}\Big)
\]
and since $\Phi$ is an isomorphism
\[
\big(\dg\A \cpf\big)\rtimes_{\hat{\alpha}}\hat{\G} = \dg\big(\A\cpf \big)\rtimes_{\hat{\alpha}}\hat{\G}.
\]
But this contradicts (\ref{comp ineq}) and the conclusion follows.
\end{proof}

%%%%%%%%%%%%%%%%%%%%%%%%%%%%%%%%%%%%%%%%%%%%%%
\section{Radically tight dynamical systems} \label{Section tight}

Trying to describe the Radical of an operator algebra is not an easy job and the literature of non-selfadjoint operator algebras is dotted with such, sometimes remarkable, efforts. We do not expect the situation to be any easier with crossed products. Quite the opposite actually and the evidence perhaps can be seen in the proof of \cite[Theorem 6.9]{KR}.  

Nevertheless our previous investigations suggest a different approach regarding the Radical. Indeed, as we saw in Corollary~\ref{Radideal} of the previous section, even though we have $\Rad\big(\A \cpf\big)= \J_{\fG}\cpf$, in general $\J_{\fG} \neq \Rad \A$. This leads to the following definition.

\begin{definition}
A dynamical system $(\A, \G, \alpha)$ where $\G$ is a locally compact group is said to be {\em radically tight} if and only if 
\[
\Rad(\A \rtimes_\alpha \G) = (\Rad \A) \rtimes \G.
\]
\end{definition}

The idea in general is that we don't attempt to describe the Radical of a crossed product but instead we try to verify one of its attributes in the hopes that it says something about the dynamical system. So far we have seen two important classes of radically tight dynamical systems:
\begin{itemize}
\item[(i)] $(\A, \G, \alpha)$ with $\A$ semisimple and $\G$ discrete abelian
\item[(ii)] $(\A, \G, \alpha)$ with $\A$ radical and $\G$ compact abelian.
\end{itemize}
The objective of this section is to provide another class of radically tight dynamical systems. This is done in Theorem~\ref{tight}.
In particular, we turn now to crossed products of TAF algebras. As was seen in \cite{KR}, this is a particularly tractable class of operator algebras to study.  

\vskip 12 pt
An approximately finite-dimensional (AF) C$^*$-algebra is the direct limit of finite-dimensional C$^*$-algebras, $\fA = \varinjlim (\fA_n, \varphi_n)$. We restrict to the case of regular embeddings, without loss of generality we will assume that matrix units are mapped to sums of matrix units. 
For each $n\in \bbN$ assume $\A_n\subset \fA_n$ is the subalgebra of upper triangular matrices and that $\varphi_n(\A_n) \subseteq \A_{n+1}$.

The subalgebra $\A = \varinjlim (\A_n, \varphi_n)$ of $\fA$ is called a {\em strongly maximal triangular AF algebra} or strongly maximal TAF for short. If $\fA_n = M_{k_n}$ and $\A_n = \T_{k_n}$ then $\fA$ is called a {\em uniformly hyperfinite (UHF) algebra} and $\A$ is called a {\em strongly maximal triangular UHF} or TUHF algebra. In this strongly maximal case, there is a well behaved diagonal $\C = \A \cap \A^*$ of $\A$ which is also given as a direct limit $\C = \varinjlim (\C_n,\varphi_n)$ where $\C_n = \A_n \cap \A_n^*$.

A detailed source for such limit algebras is Power's book \cite{Pow}. In the context of this paper, we will be working with strongly maximal TAF algebras with regular ($*$-extendable) embeddings as outlined in the previous paragraphs.

\vskip 12 pt

This class, in particular the TUHF case, is particularly tractable when dealing with crossed products since its automorphisms have a nice structure. Let 
\[
\A = \overline{\bigcup_{n\in\bbN} \T_{k_n}} \subset \overline{\bigcup_{n\in\bbN} M_{k_n}} = \fA
\] be a strongly maximal TUHF algebra with regular $*$-extendable embeddings. If $\alpha$ is a completely isometric automorphism of $\A$ then it extends to a $*$-automorphism of $\fA$ because $\fA = \cenv(\A)$, since it is simple. By the theory in \cite{Pow}, for each $n$ there is a diagonal unitary $\lambda_n \in \C$ and an integer $N\geq n$ such that $\lambda_n \alpha|_{\T_{k_n}}\lambda_n^* : \T_{k_n} \rightarrow \T_{k_N}$ is a regular $*$-extendable embedding taking matrix units to sums of matrix units. Of course, when dealing with projections we do not need to consider the action of this diagonal unitary since it goes away.

\vskip 12 pt

While the Radical still remains quite elusive in the class of TAF algebras, semisimplicity in the strongly maximal case was characterized by Donsig \cite[Theorem 4]{Don}. This condition is known as Donsig's criterion and says that a strongly maximal TAF algebra $\A$ is semisimple if and only if every matrix unit $e\in \A$ has a {\em link}, that is, $e\A e \neq \{0\}$. Conversely, a matrix unit $e\in \A$ will be called {\em linkless} if $e\A e = \{0\}$. 

In \cite{Hudson}, Hudson studied many different ideals and radicals in TAF algebras. He showed that the closed span of the linkless matrix units is in fact an ideal, called the linkless ideal, which is the smallest of all the radicals \cite[Theorem 4.1]{Hudson}. He also gave examples in the TAF class showing which radicals differed. 

In the spirit of Hudson's paper, the following theorem is an extension of Donsig's criterion for TUHF algebras and completely characterizes the Radical as the linkless ideal.

%%%%%%%
\begin{theorem}\label{TUHFRadical}
Suppose $\A$ is a strongly maximal TUHF algebra with regular $*$-extendable embeddings. $\Rad \A$ is equal to the linkless ideal.
\end{theorem}
\begin{proof}
Suppose $\A = \overline{\bigcup_{n\in \bbN} \T_{k_n}}$ with regular $*$-extendable embeddings. By \cite[Lemma 2]{Don} it is easy to see that any linkless matrix unit is in $\Rad \A$. By inductivity of closed subspaces in $\A$ \cite[Theorem 4.7]{Pow}, $\Rad \A$ is the closed span of the matrix units it contains. Thus, we need to just show that if $e$ is a matrix unit in $\Rad \A$ then there exists matrix units $e_1,\dots, e_m$ such that $e = \sum_{i=1}^{m} e_i$ and
\[
e_i\A e_i = \{0\}, \ \ 1\leq i\leq m.
\]

By contradiction assume that there is a matrix unit $e_{i_0,j_0}^{(k_m)} \in \T_{k_m} \subset \A$ such that for every $n\in\bbN$ there is a matrix unit $e_{i_n, j_n}^{(k_{m+n})} \in \T_{k_{m+n}} \subset \A$ such that
\[
e_{i_0,j_0}^{(k_m)}(e_{i_n, j_n}^{(k_{m+n})})^*e_{i_0,j_0}^{(k_m)} = e_{i_n, j_n}^{(k_{m+n})},
\]
that is $e_{i_n, j_n}^{(k_{m+n})}$ is a sub-partial isometry of $e_{i_0,j_0}^{(k_m)}$, and
\[
e_{i_n, j_n}^{(k_{m+n})}\A e_{i_n, j_n}^{(k_{m+n})} \neq \{0\}, \quad \forall n\geq 0
\]
which is the same as
\[
e_{j_n}^{(k_{m+n})}\A e_{i_n}^{(k_{m+n})} \neq \{0\}, \quad \forall n\geq 0.
\]

Now for each $n\in\bbN$ let $1\leq s_n\leq k_{m+n}$ be the smallest number such that 
\[
e_{s_n}^{(k_{m+n})}(e_{i_0,j_0}^{(k_m)})^*e_{i_0,j_0}^{(k_m)}e_{s_n}^{(k_{m+n})} = e_{s_n}^{(k_{m+n})},
\] 
that is, the first subprojection of the source projection. Similarly, let $1\leq r_n \leq k_{m+n}$ be the biggest number such that
\[
e_{r_n}^{(k_{m+n})}e_{i_0,j_0}^{(k_m)}(e_{i_0,j_0}^{(k_m)})^*e_{r_n}^{(k_{m+n})} = e_{r_n}^{(k_{m+n})},
\]
the last subprojection of the range projection. In particular, this gives that $s_n \leq j_n$ and $r_n \geq i_n$.
Combining this we get that 
\[
e_{s_n}^{(k_{m+n})}\A e_{r_n}^{(k_{m+n})} \subseteq e_{s_n}^{(k_{m+n})}e_{s_n,j_n}^{(k_{m+n})}e_{j_n}^{(k_{m+n})}\A e_{i_n}^{(k_{m+n})} e_{i_n,r_n}^{(k_{m+n})}e_{r_n}^{(k_{m+n})} \neq \{0\}.
\]

Let $T_0 = e_{i_0,j_0}^{(k_m)}$. Now, $T_0\A T_0\neq \{0\}$ and by the inductivity of this closed subspace  it is generated by its matrix units. Thus, there exists a matrix unit $S_1$ such that $T_1 = T_0S_1T_0$ is a matrix unit in $\T_{k_{m+n_1}}\subset \A$. Moreover, $e_{b_1}^{(k_{m+n_1})} = T_1^*T_1$ is a subprojection of $T_0^*T_0 = e_{j_0}^{(k_m)}$ and similarly $e_{a_1}^{(k_{m+n_1})} = T_1T_1^*$ is a subprojection of $e_{i_0}^{(k_m)}$. This gives that
\[
T_1\A T_1 \subseteq T_1e_{b_1, s_{n_1}}^{(k_{m+n_1})}e_{s_{n_1}}^{(k_{m+n_1})} \A e_{r_{n_1}}^{(k_{m+n_1})}e_{r_{n_1}, a_1}^{(k_{m+n_1})}T_1 \neq 0.
\]
In this way, recursively find matrix units $T_{l+1}, S_{l+1}$ such that $T_{l+1} = T_lS_{l+1}T_l$ for $l\geq 1$.

Donsig \cite[Theorem 3]{Don} proves that such an arrangement implies that $T_0 = e_{i_0,j_0}^{(k_m)} \notin \Rad \A$, a contradiction. Namely, he shows that $T_0 \sum_{l\geq 1} 2^{-l}S_l$ is not quasinilpotent.
\end{proof}

One should ask if this is also true for TAF algebras in general. In fact, the Radical cannot be described this way as will be seen in the following example (cf. \cite[Example 4.8]{Hudson}).

%%%%

\begin{example}
Let $\A_n = \T_2 \oplus \bigoplus_{m=1}^{n-1} \T_4$ and define embeddings $\varphi_n : \A_n \rightarrow \A_{n+1}$ by
\[
\rho_n\left(A \oplus \bigoplus_{m=1}^{n-1} B_m\right) = A \oplus \left[\begin{array}{cc} A & 0 \\ 0 & A\end{array}\right] \oplus \bigoplus_{m=1}^{n-1} B_m.
\]
Let $\A = \varinjlim (\A_n, \varphi_n)$.
These embeddings are associated with the following Bratelli diagram

%\iffalse%%%
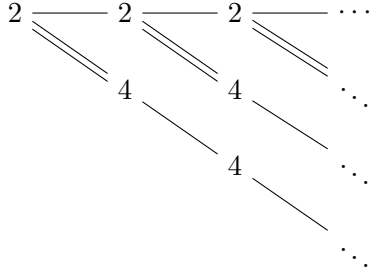
\begin{center}
\begin{tikzpicture}[node distance=0.5cm and 1cm]

\node (Col1) {2};

\node (Col2top) [right=of Col1] {2};
\node (Col2bot) [below=of Col2top] {4};

\node (Col3top) [right=of Col2top] {2};
\node (Col3mid) [below=of Col3top] {4};
\node (Col3bot) [below=of Col3mid] {4};

\node (Col4top) [right=of Col3top] {$\cdots$};
\node (Col4midtop) [right=of Col3mid] {$\ddots$};
\node (Col4midbot) [right=of Col3bot] {$\ddots$};
\node (Col4bot) [below=0.3cm of Col4midbot] {$\ddots$};

\draw (Col1) -- (Col2top);
\draw (Col2top) -- (Col3top);
\draw (Col3top) -- (Col4top);

\draw[transform canvas={yshift=-1.5pt}] (Col1) -- (Col2bot);
\draw[transform canvas={yshift=1.5pt}] (Col1) -- (Col2bot);

\draw (Col2bot) -- (Col3bot);
\draw (Col3bot) -- (Col4bot);

\draw[transform canvas={yshift=-1.5pt}] (Col2top) -- (Col3mid);
\draw[transform canvas={yshift=1.5pt}] (Col2top) -- (Col3mid);

\draw (Col3mid) -- (Col4midbot);

\draw[transform canvas={yshift=-1.5pt}] (Col3top) -- (Col4midtop);
\draw[transform canvas={yshift=1.5pt}] (Col3top) -- (Col4midtop);

\end{tikzpicture}
\end{center}
%\fi%%%

Consider now $e\in \A$ which is equal to $\begin{sbmatrix} 0 & 1\\ 0 & 0\end{sbmatrix} \in \A_1 \hookrightarrow \A$. 
For every $n\geq 1$ we have that 
\[
\varphi_{n-1}\circ\cdots \circ \varphi_1(\begin{sbmatrix} 0 & 1\\ 0 & 0\end{sbmatrix} ) = \begin{sbmatrix} 0 & 1\\ 0 & 0\end{sbmatrix}  \oplus \bigoplus_{m=1}^{n-1} \begin{sbmatrix} 0 & 1 & 0 & 0 \\ 0&0&0&0 \\ 0&0&0&1 \\ 0&0&0&0\end{sbmatrix}
\]
and so
\[
(\varphi_{n-1}\circ\cdots \circ \varphi_1(\begin{sbmatrix} 0 & 1\\ 0 & 0\end{sbmatrix} ))^2 = 0\in \A_n
\]
Thus, $e^2 = \varinjlim \varphi_{n-1}\circ\cdots\circ\varphi_1(\begin{sbmatrix} 0 & 1\\ 0 & 0\end{sbmatrix})^2 = 0$.

Suppose now for a fixed $n$ we have $b = \tilde b\in \A_n\hookrightarrow \A$. Then 
\[
(\varphi_{n-1}\circ\cdots \circ \varphi_1(\begin{sbmatrix} 0 & 1\\ 0 & 0\end{sbmatrix}))\tilde b
= \begin{sbmatrix} 0 & *\\ 0 & 0\end{sbmatrix} \oplus \bigoplus_{m=1}^{n-1} \begin{sbmatrix} 0 & * & * & * \\ 0&0&0&* \\ 0&0&0&* \\ 0&0&0&0\end{sbmatrix}
\]
and so $\left((\varphi_{n-1}\circ\cdots \circ \varphi_1(\begin{sbmatrix} 0 & 1\\ 0 & 0\end{sbmatrix}))\tilde b\right)^3 = 0\in \A_n$. In a similar manner as before one can show that 
\[
(eb)^3 = \varinjlim \left(\varphi_{n+k}\circ\cdots\circ\varphi_{n+1}\left( (\varphi_{n-1}\circ\cdots \circ \varphi_1(\begin{sbmatrix} 0 & 1\\ 0 & 0\end{sbmatrix})\tilde b \right)\right)^3 = 0.
\]

Therefore, $(eb)^3 = 0$ for every $b\in \A$ and so $e\in \Rad \A$, but $e$ can never be written as a finite sum of linkless matrix units.
\end{example}

%%%%

We now need a few lemmas before proceeding to show the main result of this section, that $(\A, \G, \alpha)$ is radically tight in the case of a strongly maximal TUHF algebra and a discrete abelian group. The following lemma is internally proved in \cite[Theorem 6.9]{KR} but we will provide the proof below for convenience.

%%%%%%%%
\begin{lemma}\label{Links}
Let $(\A, \G, \alpha)$ be a dynamical system with $\A$ a strongly maximal TAF algebra and $\G$ a discrete abelian group. If $a\notin \Rad(\A\rtimes_\alpha \G)$ then there exists $g\in \G$ such that
\[
a\A\alpha_g(a) \neq \{0\}.
\]
\end{lemma}
\begin{proof}
Fix $a\notin \Rad(\A \rtimes_\alpha G)$. Suppose, by contradiction, that for all $g\in \G$ we have that $a\A\alpha_g(a) = \{0\}$. Then for any $n\in\bbN$, $b_1, \dots, b_n \in \A$ and $g_1,\dots, g_n\in G$ one has
\begin{align*}
\left(a\sum_{i=1}^n b_i U_{g_i}\right)^2 & = \sum_{i,j=1}^n ab_iU_{g_i}ab_jU_{g_j}
\\ & = \sum_{i,j=1}^n ab_i \alpha_{g_i}(a)U_{g_i}b_jU_{g_j}
\\ & = \sum_{i,j=1}^n (0)U_{g_i}b_jU_{g_j} = 0.
\end{align*}
Therefore, by continuity, $(ac)^2 = 0$ for every $c\in \A \rtimes_\alpha \G$ and so $a\in \Rad(\A \rtimes_\alpha \G)$, a contradiction.
\end{proof}

%%%%
\begin{lemma}\label{remainsinradical}
Let $(\A, \G, \alpha)$ be a dynamical system with $\A$ a strongly maximal TUHF algebra and $\G$ a discrete abelian group. If $e_{i,j}^{(n)} \in \T_n \hookrightarrow \A$ is a matrix unit that is not in the Radical then for any $N>n$ such that 
\[
e_{i,j}^{(n)} = \sum_{k=1}^{N/n} e_{i_k,j_k}^{(N)} \in \T_N \hookrightarrow \A
\]
and $I = \max\{i_1,\dots, i_{N/n}\}$ and $J = \min\{j_1,\dots, j_{N/n}\}$, then 
\[
e_{I,J}^{(N)} \notin \Rad(\A \rtimes_\alpha \G).
\]
\end{lemma}
\begin{proof}
By contradiction, assume that $e_{I,J}^{(N)} \in \Rad(\A \rtimes_\alpha \G)$ then for $1\leq k\leq N/n$
\[
e_{i_k, j_k}^{(N)} = e_{i_k, I}^{(N)} e_{I, J}^{(N)}e_{J, j_k}^{(N)} \in \Rad(\A \rtimes_\alpha \G),
\]
since it is an ideal. 
Thus, $e_{i,j}^{(n)} = \sum_{k=1}^{N/n} e_{i_k,j_k}^{(N)} \in \Rad(\A \rtimes_\alpha \G)$, a contradiction.
\end{proof}

Now we turn to a lemma that originates from \cite{Ram} and allows us to do index chasing arguments.

%%%%%
\begin{lemma}\label{EmbeddingOrder}
Let $\varphi : \T_n \rightarrow \T_m$ be a unital regular $*$-extendable embedding. If $e_i^{(n)}\in \T_n$ is the $(i,i)$ matrix unit then $\varphi(e_i) = \sum_{j=1}^{m/n} e_{i_j}^{(m)}$ such that $i_1 < \dots < i_{m/n}$ with 
\[
i_1 \leq (i-1)m/n + 1 \ \ \textrm{and} \ \ i_{m/n} \geq im/n.
\]
\end{lemma}
\begin{proof}
We will need many more indices than in the statement of the lemma. Namely, by the regularity of $\varphi$, for every $1\leq k\leq n$, let $I_k = \{i_{k,1} < i_{k,2} < \cdots < i_{k, m/n}\} \subset \{1,\dots, m\}$ such that 
\[
\varphi(e_k^{(n)}) = \sum_{r=1}^{m/n} e_{i_{k,r}}^{(m)}.
\]
As well, for every $1\leq k < l\leq n$, there are indices $I_{k,l} = \{i_{(k,l), 1}, \dots, i_{(k,l), m/n}\}$ and $J_{k,l} = \{j_{(k,l), 1}, \dots, j_{(k,l), m/n}\}$, $m/n$ element subsets of $\{1,\dots, m\}$, such that
\[
\varphi(e_{k,l}^{(n)}) = \sum_{r=1}^{m/n} e_{i_{(k,l),r}, j_{(k,l), r}}^{(m)}.
\]

While one does not specify an order of the indices in the $I_{k,l}$ and $J_{k,l}$ subsets we do have that 
since $\varphi(e_{k,l}^{(n)})$ is strictly upper triangular then 
\begin{equation}\label{Eq:strictlyupper}
i_{(k,l),r} <  j_{(k,l), r}, \quad 1\leq r\leq m/n.
\end{equation}
Now, since $\varphi(e_{k,l}^{(n)})^*\varphi(e_{k,l}^{(n)}) = \varphi(e_l^{(n)})$ then $J_{k,l} = I_l$ and similarly since $\varphi(e_{k,l}^{(n)})\varphi(e_{k,l}^{(n)})^* = \varphi(e_k^{(n)})$ then $I_{k,l} = I_k$.
So for any $1\leq r\leq m/n$ we have that $i_{l,r} > i_{l, s}$ for $1\leq s < r$ by definition and each of these $i_{l,s}$ is strictly bigger than a unique index in $I_k$ by (\ref{Eq:strictlyupper}).
Thus, $i_{l,r}$ is strictly bigger than $r$ indices in $I_k$ and so is bigger than the $r$th biggest index, that is,
\begin{equation}\label{Eq:indexordering}
i_{k,r} < i_{l,r}, \quad 1\leq r\leq m/n.
\end{equation}

Fix $1\leq k\leq n$. Repeated uses of (\ref{Eq:indexordering}) gives that
\[
i_{k,1} < i_{k+1,1} < \cdots < i_{n,1}
\]
which in turn implies that 
\[
i_{k,1} \leq i_{s,r}, \quad k\leq s\leq n, 1\leq r\leq m/n.
\]
So $i_{k_1}$ must appear before $(n-k+1)m/n = n - (k-1)m/n$ indices, including itself, which means that
\[
i_{k_1} \leq (k-1)m/n + 1.
\]

Similarly, repeated uses of Equation \ref{Eq:indexordering} give that
\[
i_{k,m/n} > i_{k-1, m/n} > \cdots > i_{1,m/n}
\]
which in turn imply that
\[
i_{k,m/n} \geq i_{s,r}, \quad 1\leq s\leq k, 1\leq r\leq m/n.
\]
Therefore, $i_{k,m/n} \geq km/n$.
\end{proof}

The last lemma is a technical argument that sets up a contradiction in the theorem. These proofs have their roots in and Theorem \ref{tight} supersedes \cite[Theorem 6.12]{KR}. 

%%%%%%%%
\begin{lemma}\label{Technical}
Let $(\A, \G, \alpha)$ be a dynamical system with $\A$ a strongly maximal TUHF algebra with regular $*$-extendable embeddings and $\G$ a discrete abelian group. Suppose there is a matrix unit $e_{i,j}^{(n)}\notin \Rad(\A \rtimes_\alpha \G)$ such that $e_{i,j}^{(n)}\A e_{i,j}^{(n)} = \{0\}$. Then there exists an $N\in \mathbb N$ and indices $1\leq k<l<m\leq N$ such that 
\[
e_m^{(N)}\A e_l^{(N)} = e_l^{(N)}\A e_k^{(N)} = \{0\}
\]
and there exists a $g\in G$ such that 
\[
e_m^{(N)}\A\alpha_g(e_k^{(N)}) \neq \{0\}.
\]
\end{lemma}

\begin{proof}
There are two major parts to this proof. First, producing all of the relevant indices and group elements, and second in showing that these indices and group elements have the desired properties.

By Lemma \ref{Links}, since $e_{i,j}^{(n)}\notin \Rad(\A \rtimes_\alpha \G)$, there exists a $g\in \G$ such that $e_{i,j}^{(n)}\A \alpha_{g}(e_{i,j}^{(n)}) \neq \{0\}$ which implies that $e_j^{(n)} \A \alpha_{g}(e_i^{(n)}) \neq \{0\}$. By inductivity there exists an $n_1 > n$ such that $e_j^{(n)} \T_{n_1} \alpha_{g}(e_i^{(n)}) \neq \{0\}$, with $n_1$ big enough such that $\alpha_g(e_i^{(n)})\in \T_{n_1}$.
So, 
\begin{equation}\label{firstindexing}
e_i^{(n)} = \sum_{s=1}^{n_1/n} e_{i_s}^{(n_1)}, \ \  \alpha_{g}(e_i^{(n)}) = \sum_{s=1}^{n_1/n} e_{i'_s}^{(n_1)}, \ \ \textrm{and} \ \ e_j^{(n)} = \sum_{s=1}^{n_1/n} e_{j_s}^{(n_1)}, 
\end{equation}
where all indices are given in increasing order.
From $e_j^{(n)}\T_{n_1} e_i^{(n)} = \{0\}$ and $e_j^{(n)} \T_{n_1} \alpha_{g}(e_i^{(n)}) \neq \{0\}$ we have that 
\begin{equation}\label{firstordering}
i_{n_1/n} < j_1 \leq i'_{n_1/n}.
\end{equation}

\vskip 6 pt
By Lemma \ref{remainsinradical}, since $e_{i,j}^{(n)}\notin \Rad(\A \rtimes_\alpha \G)$ then  $e_{i_{n_1/n}, j_1}^{(n_1)} \notin \Rad(\A \rtimes_\alpha \G)$. So, by Lemma \ref{Links} there exists $h\in G$ such that $e_{j_1}^{(n_1)}\A \alpha_{h}(e_{i_{n_1/n}}^{(n_1)}) \neq 0$. By inductivity there exists an $N > n_1$ such that $e_{j_1}^{(n_1)}\T_{N} e_{i_{n_1/n}}^{(n_1)} \neq \{0\}$ with $N$ big enough so that $\alpha_h(e_{i_{n_1/n}}^{(n_1)})$ and $\alpha_{h}(e_{j_1}^{(n)})$ are in $\T_N$.
Now, 
\begin{equation}\label{secondindexing}
\begin{split}
e_{i_{n_1/n}}^{(n_1)} &= \sum_{s=1}^{N/n_1} e_{(i_{n_1/n})_s}^{(N)}, 
\\ e_{j_1}^{(n_1)} &= \sum_{s=1}^{N/n_1} e_{(j_1)_s}^{(N)},
\end{split}
\quad\quad \textrm{and}\quad\quad
\begin{split}
\alpha_{h}(e_{i_{n_1/n}}^{(n_1)}) & = \sum_{s=1}^{N/n_1} e_{(i_{n_1/n})'_s}^{(N)},
\\  \alpha_{h}(e_{j_1}^{(n)}) & = \sum_{s=1}^{N/n_1} e_{(j_1)'_s}^{(N)},
\end{split}
\end{equation}
where all indices are given in increasing order.

By hypothesis, $e_j^{(n)}\A e_i^{(n)} = \{0\}$ which implies that
\begin{equation}\label{firstnolink}
e_{j_1}^{(n_1)}\A e_{i_{n_1/n}}^{(n_1)} = \{0\} \quad \textrm{and} \quad \alpha_{h}(e_{j_1}^{(n_1)})\A \alpha_{h}(e_{i_{n_1/n}}^{(n_1)}) = \{0\}
\end{equation}
since there cannot be any links between subprojections of unlinked projections. In particular,
\[
e_{j_1}^{(n_1)}\T_N e_{i_{n_1/n}}^{(n_1)} = \{0\} \quad \textrm{and} \quad \alpha_{h}(e_{j_1}^{(n_1)})\T_N \alpha_{h}(e_{i_{n_1/n}}^{(n_1)}) = \{0\}
\]
which implies by (\ref{secondindexing}) that 
\begin{equation}\label{secondordering}
(i_{n_1/n})_{N/n_1} < (j_1)_1 \quad  \textrm{and} \quad (i_{n_1/n})'_{N/n_1} < (j_1)'_1.
\end{equation}
Furthermore, by the definition of $h$ we have that $e_{j_1}^{(n_1)}\T_{N} \alpha_h(e_{i_{n_1/n}}^{(n_1)}) \neq \{0\}$ which by (\ref{secondindexing}) implies that 
\begin{equation}\label{thirdordering}
(j_1)_1 \leq (i_{n_1/n})'_{N/n_1}.
\end{equation}
By Lemma \ref{EmbeddingOrder} applied to $e_{j_1}^{(n_1)}$ and the regular embedding $\T_{n_1} \hookrightarrow T_N$ we have that 
\begin{equation}\label{fourthordering}
 j_1N/n_1 \leq (j_1)_{N/n_1}.
\end{equation}
Similarly, Lemma \ref{EmbeddingOrder} applied to $e_{j_1}^{(n_1)}$ and the regular embedding $\alpha_h: \T_{n_1} \rightarrow \T_N$ we have that 
\begin{equation}\label{fifthordering}
(j_1)'_1 \leq (j_1-1)N/n_1 + 1
\end{equation}

Combining (\ref{secondordering}), (\ref{thirdordering}), (\ref{fourthordering}), and (\ref{fifthordering})
\begin{equation}\label{mainordering}
(i_{n_1/n})_{N/n_1} < (j_1)_1 \leq (i_{n_1/n})'_{N/n_1} < (j_1)'_1 < (j_1)_{N/n_1}.
\end{equation}
This ordering of indices is what will allow us to produce the desired results of the lemma.

\vskip 6 pt
Let $N$ and $g$ be as above and let $k = (i_{n_1/n})_{N/n_1}, l = (j_1)_1$, and $m = (j_1)_{N/n_1}$. 

\vskip 6 pt
\noindent {\bf \em Claim 1: $e_m^{(N)}\A e_l^{(N)} = e_l^{(N)}\A e_k^{(N)} = \{0\}$} 
\vskip 4 pt

By (\ref{firstnolink}), namely $e_{j_1}^{(n_1)}\A e_{i_{n_1/n}}^{(n_1)} = \{0\} = \alpha_{h}(e_{j_1}^{(n_1)})\A \alpha_{g}(e_{i_{n_1/n}}^{(n_1)})$, and (\ref{secondindexing}) we have that
\[
 e_l^{(N)}\A e_k^{(N)} = e_{(j_1)_1}^{(N)}\A e_{(i_{n_1/n})_{N/n_1}}^{(N)} = \{0\} \quad \textrm{and} \quad e_{(j_1)'_1}^{(N)}\A e_{(i_{n_1/n})'_{N/n_1}}^{(N)} = \{0\}.
\]
And so, by (\ref{mainordering}) and the previous equation
\[
 e_{(j_1)'_1, (j_1)_{N/n_1}}^{(N)}\left(e_{(j_1)_{N/n_1}}^{(N)} \A e_{(j_1)_1}^{(N)} \right)e_{(j_1)_1, (i_{n_1/n})'_{N/n_1}}^{(N)} 
 \subseteq  e_{(j_1)'_1}^{(N)}\A e_{(i_{n_1/n})'_{N/n_1}}^{(N)} = \{0\}.
\]
Thus, 
\[
e_m^{(N)} \A e_l^{(N)} = e_{(j_1)_{N/n_1}}^{(N)} \A e_{(j_1)_1}^{(N)} = \{0\}
\]
and the claim is established.

\vskip 12 pt
\noindent {\bf\em Claim 2: $e_m^{(N)} \A \alpha_g(e_k^{(N)}) \neq \{0\}$.}
\vskip 4 pt
To this end, there exists an $n_3 \geq N$ such that $\alpha_{g}(e_{(i_{n_1/n})_{N/n_1}}^{(N)}) \in \T_{n_3}$.
Again we need to index some matrix units in $\T_{n_3}$
\begin{align}\label{thirdindexing}
\alpha_{g}(e_{(i_{n_1/n})_{N/n_1}}^{(N)}) &= \sum_{s=1}^{n_3/N} e_{((i_{n_1/n})_{N/n_1})'_s}^{(n_3)} \nonumber
\\ e_{(j_1)_{N/n_1}}^{(N)} &= \sum_{s=1}^{n_3/N} e_{((j_1)_{N/n_1})_s}^{(n_3)}
\\ e_{i'_{n_1/n}}^{(n_1)} &= \sum_{s=1}^{n_3/n_1} e_{(i'_{n_1/n})_s}^{(n_3)}. \nonumber
\end{align}
By (\ref{firstindexing}) and (\ref{secondindexing}) we have
\[
\sum_{s=1}^{n_1/n} e_{i'_s}^{(n_1)} = \alpha_{g}(e_i^{(n)}) = \alpha_g(\sum_{t=1}^{n_1/n} e_{i_t}^{(n_1)})= \alpha_{g}\big(\sum_{s=1}^{N/n_1}\sum_{t=1}^{n_1/n} e_{(i_t)_s}^{(N)}\big)
\] 
and the very last index of each summation when embedded into $\T_{n_3}$ must be equal, namely by (\ref{thirdindexing})
\[
((i_{n_1/n})_{N/n_1})'_{n_3/N} = (i'_{n_1/n})_{n_3/n_1}.
\]
Now, $j_1 \leq i'_{n_1/n}$ by (\ref{firstordering}) and so $e_{j_1, i'_{n_1/n}}^{(n_1)} \in \T_{n_1}$. Thus, after embedding this partial isometry into $\T_{n_3}$, ignoring any diagonal unitary element picked up from $\alpha_g$, we have that $((j_1)_{N/n_1})_1 < (i'_{n_1/n})_{n_3/n_1}$, that is, the first subprojection of $e_{j_1}^{(n_1/n)}$ precedes the last subprojection of $e_{i'_{n_1/n}}^{(n_1/n)}$.
In other words,
\[
e_{((j_1)_{N/n_1})_1, ((i_{n_1/n})_{N/n_1})'_{n_3/N}}^{(n_3)} = e_{((j_1)_{N/n_1})_1, (i'_{n_1/n})_{n_3/n_1}}^{(n_3)} \in \T_{n_3}.
\]
Therefore, by (\ref{thirdindexing})
\[
e_m^{(N)} \A \alpha_g(e_k^{(N)}) = e_{(j_1)_{N/n_1}}^{(N)}\A \alpha_{g}(e_{(i_{n_1/n})_{N/n_1}}^{(N)}) \neq \{0\}
\]
because $e_{((j_1)_{N/n_1})_1, ((i_{n_1/n})_{N/n_1})'_{n_3/N}}^{(n_3)} \in e_{(j_1)_{N/n_1}}^{(N)}\A \alpha_{g}(e_{(i_{n_1/n})_{N/n_1}}^{(N)})$.
\end{proof}

%%%%%%%%
\begin{theorem} \label{tight}
Let $(\A, \G, \alpha)$ be a dynamical system with $\G$ a discrete abelian group and $\A$ a strongly maximal TUHF algebra with regular $*$-extendable embeddings. Then $(\A,\G,\alpha)$ is radically tight.
\end{theorem}

\begin{proof}
By Corollary \ref{Radideal}, the result will be established if we can show that $\Rad \A = \Rad(\A \rtimes_\alpha \G) \cap \A$.

By way of contradiction, assume that the above equality fails. By inductivity,  there is a matrix unit $e\in \Rad \A \cap \T_n$ such that $e\notin \Rad(\A\rtimes_\alpha \G)$. By Theorem \ref{TUHFRadical} there exists an $m \geq n$ such that if $e = \sum_{i=1}^{m/n} e_i \in \T_m\subset \A$ then
$e_i\A e_i = \{0\}$, for $1\leq i\leq m/n$. If all $e_i \in \Rad(\A\rtimes_\alpha \G)$ then so is $e$. Thus, there exists $1\leq i\leq m/n$ such that $e_i \notin \Rad(\A \rtimes_\alpha \G)$.

\vskip 6 pt
Starting over with the notation, we know now that there exists a matrix unit $e_{i,j}^{(n)} \in \T_n \subset \A$ such that $e_{i,j}^{(n)}\A e_{i,j}^{(n)} = \{0\}$ and $e_{i,j}^{(n)}\notin \Rad(\A \rtimes_\alpha \G)$. 

By Lemma \ref{Technical} there exists an $n_1\in \mathbb N$ and indices $1\leq k<l<m\leq n_1$ such that 
\[
e_m^{(n_1)}\A e_l^{(n_1)} = e_l^{(n_1)}\A e_k^{(n_1)} = \{0\}
\]
and there exists a $g\in G$ such that $e_m^{(n_1)}\A\alpha_g(e_k^{(n_1)}) \neq \{0\}$.
By inductivity, there exists $n_2 > n_1$ such that
$e_{m}^{(n_1)}\T_{n_2}\alpha_{g}(e_k^{(n_1)}) \neq \{0\}$. 
Now 
\begin{equation}\label{Tindices}
\begin{split}
\alpha_{g}(e_k^{(n_1)}) &= \sum_{s=1}^{n_2/n_1} e_{k'_s}^{(n_2)},  \\
e_l^{(n_1)} &= \sum_{s=1}^{n_2/n_1} e_{l_s}^{(n_2)}, 
\end{split}
\quad\quad\textrm{and}\quad\quad
\begin{split}
e_{m}^{(n_1)} &= \sum_{s=1}^{n_2/n_1} e_{m_s}^{(n_2)},
\\ \alpha_{g}(e_l^{(n_1)}) &= \sum_{s=1}^{n_2/n_1} e_{l'_s}^{(n_2)},
\end{split}
\end{equation}
where the indices are in increasing order.
Since $e_{m}^{(n_1)}\A e_l^{(n_1)} = \{0\}$ then from (\ref{Tindices}) we get that 
\begin{equation}\label{Tfirstordering}
l_{n_2/n_1} < m_1.
\end{equation} 
Similarly, 
\[
\alpha_{g}(e_l^{(n_1)})\A \alpha_{g}(e_k^{(n_1)}) = \alpha_g(e_l^{(n_1)}\A e_k^{(n_1)}) = \alpha_g(\{0\}) = \{0\}
\]
and so looking at the indices in (\ref{Tindices}) we get
\begin{equation}\label{Tsecondordering}
k'_{n_2/n_1} < l'_1.
\end{equation}
Lastly, $e_{m}^{(n_1)}\T_{n_2}\alpha_{g}(e_k^{(n_1)}) \neq \{0\}$ and so (\ref{Tindices}) gives that
\begin{equation}\label{Tthirdordering}
m_1 \leq k'_{n_2/n_1}.
\end{equation}
Using Lemma \ref{EmbeddingOrder} on $e_l^{(n_1)}$ with embedding $\T_{n_1} \hookrightarrow T_{n_2}$ 
gives, by (\ref{Tindices}) that
\begin{equation}\label{Tfourthordering}
l n_2/n_1 \leq l_{n_2/n_1}.
\end{equation}
And similarly, using Lemma \ref{EmbeddingOrder} on $e_l^{(n_1)}$ with embedding $\alpha_g : \T_{n_1} \rightarrow \T_{n_2}$ gives, by (\ref{Tindices}), that
\begin{equation}\label{Tfifthordering}
l'_1 \leq (l-1)n_2/n_1 + 1
\end{equation}
Finally, combining (\ref{Tfourthordering}), (\ref{Tfirstordering}), (\ref{Tthirdordering}), (\ref{Tsecondordering}), and (\ref{Tfifthordering}) gives that
\[
l n_2/n_1 \leq l_{n_2/n_1} < m_1 \leq k'_{n_2/n_1} < l'_1 \leq (l-1)n_2/n_1 + 1,
\]
which is a contradiction.
Therefore, $\Rad \A = \Rad(\A \rtimes_\alpha \G) \cap \A$ and $(\A, \G, \alpha)$ is radically tight.
\end{proof}

Combining Theorems \ref{TUHFRadical} and \ref{tight} we get.
%%%
\begin{corollary}
Let $(\A, \G, \alpha)$ be a dynamical system wit $\G$ a discrete abelian group and $\A$ a strongly maximal TUHF algebra with regular $*$-extendable embeddings. Then 
\[
\Rad(\A \rtimes_\alpha \G) = \overline\spn\left\{eU_g : e\in \A, e\A e = \{0\}, g\in \G\right\}.
\]
\end{corollary}

\begin{remark}
One should note that the Radical being equal to the linkless ideal for a strongly maximal TAF algebra does not imply that any dynamical system it is in with a discrete abelian group is radically tight. A counterexample to this is \cite[Example 6.8]{KR}.
\end{remark}

%%%%%%%%%%%%%%%%
%%%%%%%%%%%%%%%%%%%%%%%%%%%%%%%%

\end{document}